\documentclass[11pt, reqno]{amsart}

\usepackage{amssymb}
\usepackage{graphicx}
\usepackage{amsfonts}
\usepackage{amsthm}
\usepackage{amsmath}
\usepackage{empheq}
\usepackage{algorithm}
\usepackage{algorithmic}
\usepackage[T1]{fontenc}
\usepackage{xcolor}
\usepackage{bbm}

\usepackage{subcaption}
\usepackage{epsfig}
\usepackage{epstopdf}

\usepackage[left=1in, right=1in]{geometry}

\usepackage[bookmarksnumbered=true, citecolor=blue, colorlinks=true]{hyperref}
\usepackage{cleveref}

\usepackage[numbers,sort&compress]{natbib}

\theoremstyle{plain}
\newtheorem{theorem}{Theorem}[section]
\newtheorem{lemma}[theorem]{Lemma}
\newtheorem{corollary}[theorem]{Corollary}
\newtheorem{assumption}[theorem]{Assumption}
\theoremstyle{definition}

\newtheorem{example}[theorem]{Example}

\newtheorem{proposition}[theorem]{Proposition}

\theoremstyle{remark}

\numberwithin{equation}{section}

\newcommand{\Prox}{\text{Prox}}

\newcommand{\bR}{\mathbb R}

\newcommand{\cG}{\mathcal{G}}

\newcommand{\vb}{\boldsymbol b}
\newcommand{\vu}{\boldsymbol u}
\newcommand{\vw}{\boldsymbol w}
\newcommand{\vx}{\boldsymbol x}
\newcommand{\vy}{\boldsymbol y}
\newcommand{\vz}{\boldsymbol z}

\newcommand{\mA}{\mathbf A}
\newcommand{\mPhi}{\pmb \Phi}
\newcommand{\mB}{\mathbf B}
\newcommand{\mH}{\mathbf H}
\newcommand{\mI}{\mathbf I}
\newcommand{\mL}{\mathbf L}
\newcommand{\mW}{\mathbf W}
\newcommand{\mX}{\mathbf X}
\newcommand{\mZ}{\mathbf Z}

\begin{document}

\title[Divide-and-Conquer Algorithm]{A Divide-and-Conquer  Algorithm for Distributed Optimization on Networks}

%    Information for first author
\author{Nazar Emirov}
%    Address of record for the research reported here
\address{Department of Computer Science, Boston College, Chestnut Hill, Massachusetts 02467}
%    Current address
%\curraddr{Department of Mathematics and Statistics,
%Case Western Reserve University, Cleveland, Ohio 43403}
\email{nazar.emirov@bc.edu}
%    \thanks will become a 1st page footnote.
%\thanks{The first author was supported in part by NSF Grant \#000000.}

\author{Guohui Song}
\address{Department of Mathematics and Statistics, Old Dominion University, Norfolk, Virginia 23529}
\email{gsong@odu.edu}

%    Information for second author
\author{Qiyu Sun}
\address{Department of Mathematics, University of Central Florida, Orlando, Florida 32816}
\email{qiyu.sun@ucf.edu}

\thanks{The project is partially supported by the National Science Foundation
DMS-1816313 and DMS-1939203.}

%\keywords{kk}

\begin{abstract}
In  this paper, we consider networks with topologies described by some connected undirected graph  $\cG=(V, E)$
 and with some agents (fusion centers) equipped with processing power and local peer-to-peer communication,
and optimization problem $\min_{{\vx}}\big\{F({\vx})=\sum_{i\in V}f_i({\vx})\big\}$ with local objective functions $f_i$  depending only on neighboring variables of the vertex $i\in V$. We introduce a divide-and-conquer algorithm  to solve the above  optimization problem in a distributed and decentralized manner. The proposed divide-and-conquer algorithm has exponential convergence, its computational cost is almost linear with respect to the size of the network, and it can be fully implemented at fusion centers of the network. Our numerical demonstrations also indicate that the proposed divide-and-conquer algorithm has superior performance than  popular decentralized optimization methods do for the least squares problem  with/without $\ell^1$ penalty.
 \end{abstract}

\maketitle

\section{Introduction}

Networks have been widely used in many real world applications, including  (wireless) sensor networks, smart grids,  social networks and epidemic spreading \cite{Akyildiz2002, Chong2003, Shuman2013, Ortega2018, Yick2008, Molzahn17, Hebner2017}. Their complicated topological structures could be described by some graphs with vertices representing agents and edges between two vertices indicating the availability of a  peer-to-peer communication between agents, or the functional connectivity between neural regions in brain, or the correlation between temperature records of neighboring weather stations. Graph signal processing and graph machine learning provide innovative frameworks to process and  learn data on networks. By leveraging graph  spectral theory and applied harmonic analysis, many concepts in the classical Euclidean setting have been extended to the graph setting, such as graph Fourier transform, graph wavelet transform and graph filter banks, in recent years \cite{Recht11, Kekatos13,  Giannakis16, Cheng19, Cheng17,  Chen14, Ozdaglar17}. Graph machine learning has also developed new tools, including graph representation learning, graph neural networks and geometric deep learning,  to process data on networks \cite{Bronstein2021, Bronstein2017, Hamilton2017a, Zhou2020}. In this paper, we introduce a divide-and-conquer algorithm, DAC for abbreviation, to solve the following optimization problem
\begin{equation}\label{MainOptimizationProblem}
%\tilde{\vx}=\arg
\min_{{\vx}\in \mathbb{R}^N}\Big\{F({\vx})=\sum_{i\in V}f_i({\vx})\Big\}
\end{equation}
on a connected undirected graph $\cG:=(V, E)$ of  order $N\ge 1$, where ${\vx}=(x(i))_{i\in V}$ and local objective functions $f_i$ depend only on neighboring variables of the vertex $i\in V$, see Assumption \ref{localobjectivefunction.assump}. The formulation \eqref{MainOptimizationProblem} can be seen in various applications such as wireless communication, power systems, control problems, empirical risk minimization, binary classification, etc \cite{Giannakis16, Kekatos13, Ozdaglar17, Molzahn17}.

Many networks in modern infrastructure have large amount of agents correlated with each other and the size of data set on the network  required to process is also huge. It is often impractical or even impossible to have a central server to collect and process the whole data set. Hence it is of great importance to design distributed and decentralized algorithms to solve the global optimization problems, that is, the storage and processing of data need to be distributed among agents and the local peer-to-peer communication  should be employed to handle the coordination of results from each agent instead of a central server \cite{Yang2019,Boyd11,Nedich2015,Sayed2014a,Nedic2018,Bertsekas1989}. In this paper, we consider the optimization problems on networks with a  two-layer hierarchical structure: 1) Most of agents of the network have no or very little computational power and their communication ability is only limited to transmit the data within certain range; 2) Some agents, called {\em fusion centers},  have greater computational power and could communicate with nearby  agents and neighboring fusion centers within certain distance. In other words, each fusion center is a ``central'' server for its  neighboring subnetwork of small size and there is no central server for the whole network, see Section \ref{fusioncenter.subsection}, Assumption \ref{fusioncenter.assump} and an illustrate example displayed in \Cref{FusionCenters.fig}. The above two-layer hierarchical structure on networks is between two extremes: centralized networks and  fully decentralized networks,  where the centralized network usually assumes only the central server has data processing ability and each agent sends the data to the central server for processing, while the fully decentralized network assumes that each agent has both communication and computation ability to store the data and perform the computation on their own. For a centralized network, it might require an extremely powerful central server and a high demand for each agent to transmit data to the central server. On the other hand, fully decentralized networks might not be practical in some applications in which every agent is not powerful enough to implement the local computation. Our DAC algorithm to solve the distributed optimization problem
\eqref{MainOptimizationProblem} is designed to be completely implemented on the two-layer hierarchical structure, see \Cref{Main.algorithm}.

Most of the existing decentralized methods are based on the consensus technique, where  each agent holds a local
estimate of the global variable and these local estimates  reach a consensus after certain iterations
\cite{Ozdaglar17, Wei12, Shi2014, Fazlyab18, Boyd11, Shi15, Zhang14, Xin2020}. Each agent of the network  would evaluate its  local gradient independently and communicates only with its neighbors to  update local estimate of the global variable, which makes it amenable to parallel computation and eliminates the large communication bandwidth of a central server. However, the computational cost at each agent still depends on the size of the local variable (the same size of the global variable), which indicates that  the local computation cost at each agent might still be quite expensive. In the decentralized implementation of our DAC algorithm, each agent will only update/estimate a few components of the global variable  and communicate with its neighbors to have an aggregate combination of the local solutions, and hence the subproblem for each agent has a much smaller size than the global problem, and the solutions of subproblems will be combined to get the overall solution of the global problem, see \eqref{localminimization.def0} and \Cref{Main.algorithm}.

A key difference of the proposed DAC algorithm from other decentralized methods is the size reduction of the subproblems. Most existing decentralized methods distribute the computation of gradients to each node, but the size of the subproblem at each node usually depends on the dimension of the global variable and remains the same as the global problem. The proposed DAC further distribute the estimates of the components of the global variable to each fusion center, and hence the subproblem would have a much smaller size than the global problem, see \eqref{DACalgorithm.defa}. This reduces both the storage and the computation requirements of each fusion center, and is particularly useful when the global variable has a huge number of components. Another difference is the hierarchical network structure to implement DAC. Even though we will focus on a two-layer hierarchical network structure in this paper, the proposed DAC method would be readily to extended to multiple layers. This provides flexibility in computation/communication and also makes it more scalable than many other decentralized methods. Moreover, the exponential convergence of the proposed DAC method is ensured by a thorough theoretical analysis.

The paper is organized as follows. In  Section \ref{preliminaries.section}, we recall some preliminaries on the underlying graph $\cG$ and the objective function $F$  of  the optimization problem \eqref{MainOptimizationProblem}, and we introduce fusion centers to  implement the proposed DAC algorithm. In Section \ref{DistributedOptimization.section}, we introduce a novel DAC algorithm \eqref{DACalgorithm.def} to solve the optimization problem \eqref{MainOptimizationProblem} in a distributed and decentralized manner and present a scalable implementation at fusion centers, see \Cref{Main.algorithm}. In Section \ref{convergence.section},  we discuss the exponential convergence of the iterative DAC algorithm and  provide an error estimate when the local minimization in the DAC algorithm  is inexact, see Theorems \ref{ConvergenceTheorem} and \ref{inexactConvergenceTheorem}. In Section \ref{NumericExamples}, we demonstrate the performance of the proposed DAC algorithm and make comparisons with some popular decentralized optimization methods. In Section \ref{proof.section}, we collect the proofs of all theorems.

\section{Preliminaries}  % on  Graph Structures}
\label{preliminaries.section}

In our proposed DAC algorithm to solve the optimization problem \eqref{MainOptimizationProblem}, the underlying graph $\cG=(V, E)$ is connected and undirected and it has  polynomial growth property (see \Cref{polynomialgrowth.subsection}), the global objective function  $F$ is strongly convex and each local objective function $f_i$ is smooth and dependent only on neighboring variables of the vertex $i\in V$
 (see Section \ref{objectivefunction.subsection}), and fusion centers to implement the DAC algorithm  are
 appropriately located   (see Section \ref{fusioncenter.subsection}) and equipped with enough processing power and resources (see \Cref{fusioncenter.assump} in \Cref{DistributedOptimization.section}).

\subsection{Graphs with polynomial growth property}\label{polynomialgrowth.subsection}

For a connected undirected  graph $\cG:=(V, E)$, let the  {\em geodesic distance}  $\rho(i,j)$ between vertices $i$ and $j\in V$ be  the number of edges in the shortest path to connect $i$ and $j$. Using the geodesic distance $\rho$, we denote the set of  all $R$-neighbors of a vertex $i\in V$ by
$$B(i,R)=\big\{j\in V, \ \rho(i,j)\le R\big\}, \ R\ge 0.$$

 Let $\mu_\cG$ be the {\em counting measure} on the graph  $\cG$ defined by $\mu_\cG(W)=\# W$, the cardinality of  a subset $W\subset V$. We say a connected undirected graph $\cG=(V, E)$ has  {\em polynomial growth property} if there exist positive constants  $d(\cG)$ and $D_1(\cG)$  such that
\begin{equation}\label{PolynomialGrowth}
    \mu_{\cG}(B(i,R))\le D_1(\cG)(R+1)^{d(\cG)}\ \ \text{for all}\  i\in V \text{ and } \ R\geq 0.
\end{equation}
The minimal positive constants $d(\cG) $ and $D_1(\cG)$ in \eqref{PolynomialGrowth} are known as {\em Beurling dimension} and {\em density} of the graph $\cG$ respectively \cite{Cheng17}.

In this paper, we make the following assumption on the underlying graph of our optimization problem \eqref{MainOptimizationProblem}.
\begin{assumption} \label{graphstructure.assumption}
The underlying graph $\cG$
is connected  and undirected with order $N$, and it  has polynomial growth  property with   Beurling dimension and density denoted by
 $d({\mathcal G})$ and  $D_1({\mathcal G})$ respectively.
 \end{assumption}

A stronger assumption on the graph $\cG$ than its polynomial growth property is that the counting measure $\mu_{\cG}$ is a {\em doubling measure}, i.e.,  there exists a positive number $D_0(\cG)$ such that
\begin{equation}\label{DoublingMeasure.def}
\mu_{\cG}(B(i,2R))\le D_0(\cG) \mu_{\cG}(B(i,R)) \ \ \text{ for all } i\in V \text{ and } R\ge 0
\end{equation}
\cite{Cheng17, Shin19, Yang2013}.
The smallest constant $D_0(\cG)$ in \eqref{DoublingMeasure.def} is known as the {\em doubling constant} of the counting measure. In other words, the doubling property for the counting measure indicates that for any vertex $i$, the number of $(2R)$-neighbors is at most a multiple of the number of $R$-neighbors. It is direct to observe that the doubling property \eqref{DoublingMeasure.def} for the counting measure $\mu_{\mathcal G}$ implies the polynomial growth property \eqref{PolynomialGrowth} with
$d(\cG)\le \log_2(D_0(\cG))$ and $D_1(\cG)\le D_0(\cG)$,
since for all $R\ge 1$,
$$\mu_{\cG}(B(i,R)) \le (D_0(\cG))^l\mu_{\cG}(B(i,R/2^l))=  (D_0(\cG))^l \le D_0(\cG) (R+1)^{\log_2(D_0(\cG))},
$$
where $l$ is the integer  satisfying $2^{l-1}\le R< 2^{l}$.

Our illustrative examples of connected undirected graphs are random geometric  graphs generated by the GSPToolbox, where  $N$ vertices are randomly deployed in the unit square $[0, 1]^2$ and an edge existing between two vertices if their Euclidean distance is not larger than  $\sqrt{2/N}$ \cite{Jiang19, Nathanael14}, see Figure \ref{FusionCenters.fig}.
\begin{figure}[t] %[h!]
\vspace{-.1in}
      \begin{center}
\includegraphics[width=50mm, height=55mm]{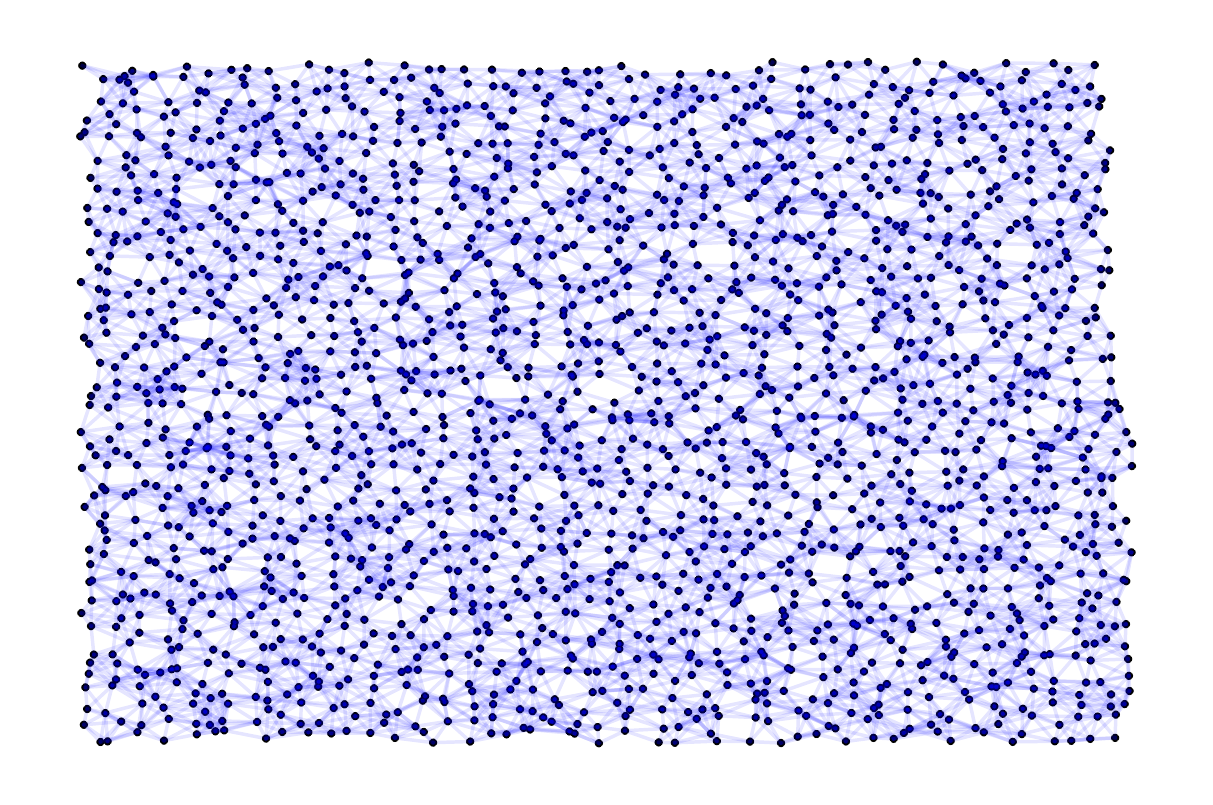}
\includegraphics[width=50mm, height=55mm]{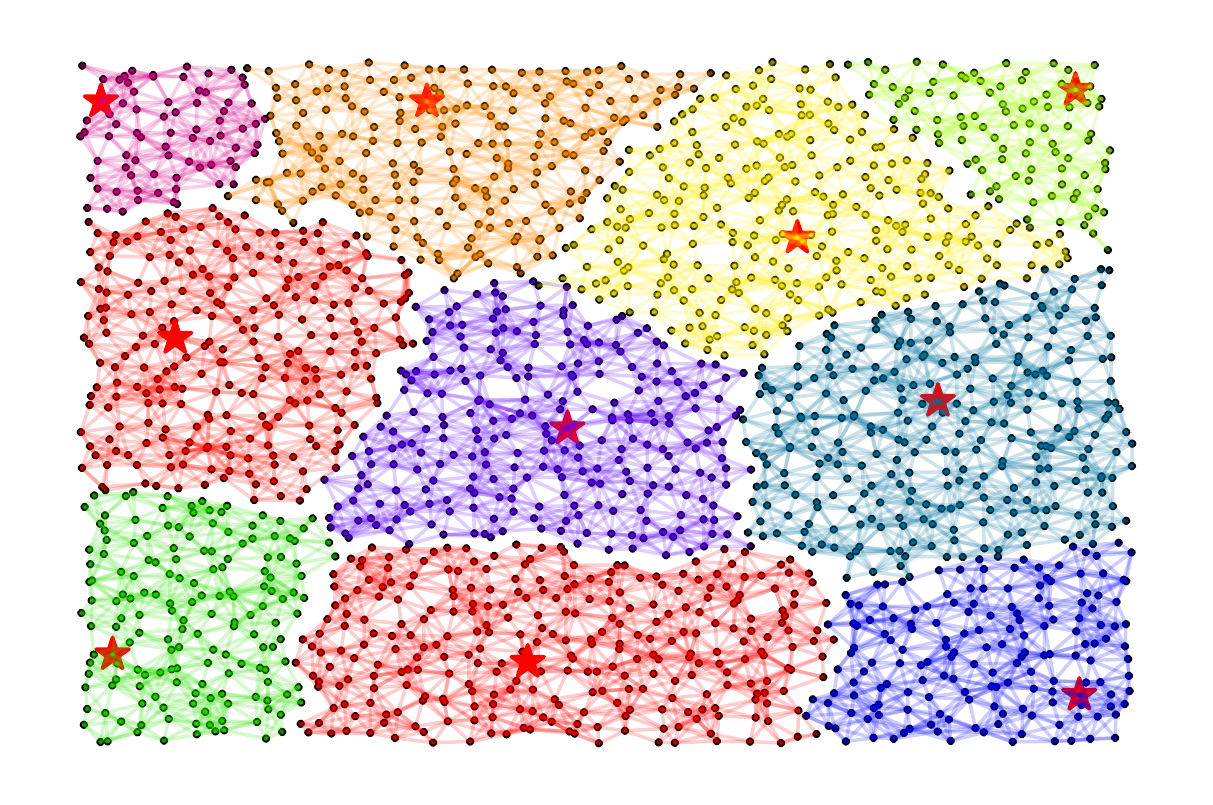}
\includegraphics[width=50mm, height=55mm]{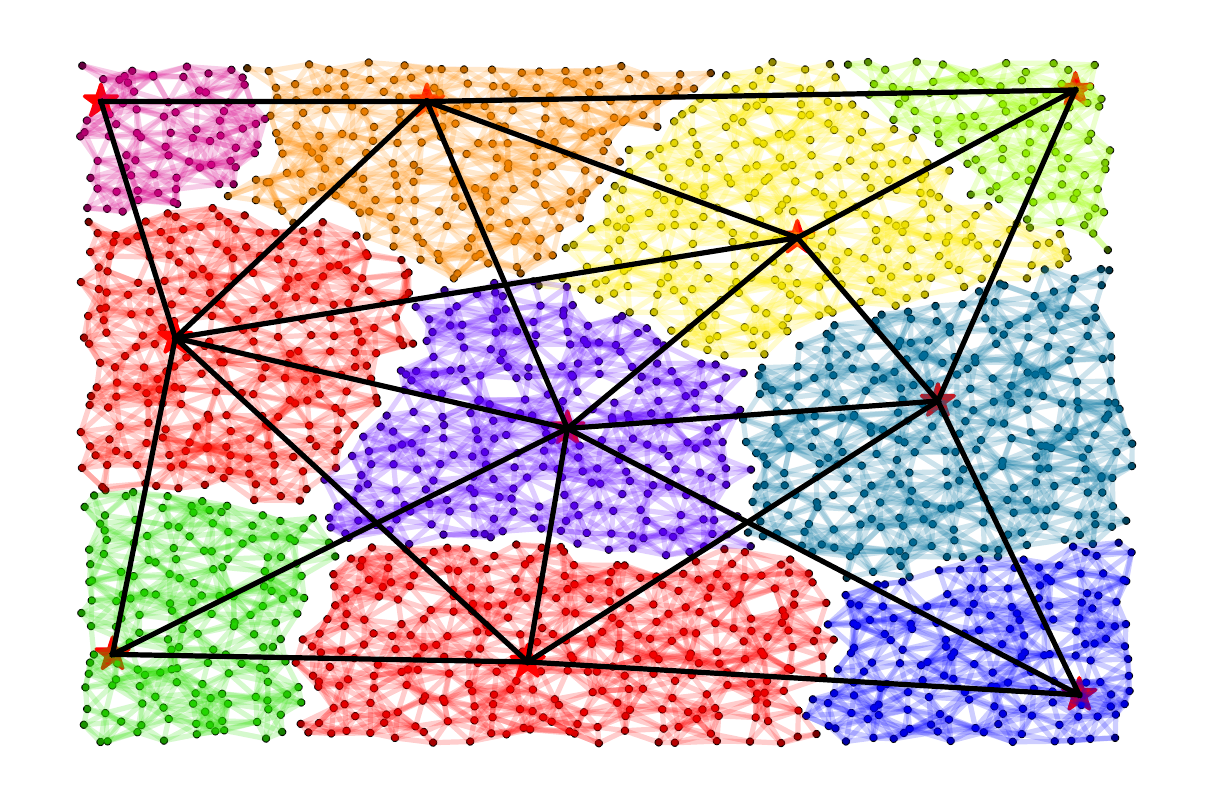}
\caption{Plotted on the  left is a  random geometric graph  ${\mathcal G}_N$ of the order $N=2048$ with  Beurling  dimension $d({\mathcal G}_{2048})=2$
and Beurling density $D_1({\mathcal G}_{2048})=7.6116$. In the middle is a family of  fusion centers marked in stars, and corresponding Voronoi diagrams $D_{\lambda}, \lambda\in\Lambda$, marked with various colors. On the right is the communication network for the fusion centers, marked by black solid line, where fusion center $\lambda$ communicates with $\lambda'$ if  $D_{\lambda}\cap D_{\lambda',R,2m}\neq \emptyset$ with $R=3$ and $m=1$, see Assumption \ref{fusioncenter.assump}.}
\label{FusionCenters.fig}
\end{center}
\vspace{0.2in}
\end{figure}

\subsection{Local and global objective functions}\label{objectivefunction.subsection}

In this paper,  the local objective functions $f_i, i\in V$, in the optimization problem \eqref{MainOptimizationProblem} are assumed to be smooth and dependent only on neighboring variables.

\begin{assumption}\label{localobjectivefunction.assump} For each $i\in V$, the local objective function  $f_i(\vx)$ is continuously differentiable and it depends only on  $x(j), j\in B(i, m)$,  where $m\ge 1$ and $\vx=(x(j))_{j\in V}$.
\end{assumption}

For a  matrix ${\mA}=(a(i,j))_{i,j\in V}$ on the connected undirected graph $\cG=(V, E)$, we define its {\em geodesic-width} $\omega({\mA})$ to be the smallest nonnegative integer such that
\begin{equation*}
a(i,j)=0 \ \ {\rm  for\ all} \ \  i,j\in V \ \ {\rm with} \  \  \rho(i,j)>\omega({\mA}).
\end{equation*}
Then the neighboring variable dependence of the objective functions $f_i, i\in V$, in
 \Cref{localobjectivefunction.assump} can be described  by a geodesic-width requirement for the gradient of the vector-valued function $\vx\to (f_i(\vx))_{i\in V}$ on ${\mathbb R}^N$,
\begin{equation}\label{assumption2.characterization}
\omega\Big( \Big(\frac{\partial f_i(\vx)}{\partial x(j)}\Big)_{i, j\in V}\Big)\le m \ \ {\rm for \ all}\  \vx\in {\mathbb R}^N.
\end{equation}
Due to the above characterization, we use $m$  to denote the {\em neighboring radius} of the local objective functions $f_i, i\in V$. In the classical least squares problem $\min_{\vx\in {\mathbb R}^N}  \|\mA \vx-\vb\|_2^2$
associated with the measurement matrix  ${\bf A}=(a(i,j))_{i,j\in V}$ having geodesic-width $m$ and the noisy observation data  ${\vb}=(b(i))_{i\in V}\in {\mathbb R}^N$,  one may verify that the local objective functions
$$f_i(\vx)=\Big(\sum_{j\in B(i,m)} a(i,j) x(j)-b(i)\Big)^2,\  i\in V,  $$
satisfy Assumption \ref{localobjectivefunction.assump} and have neighboring radius $m$.

For two square matrices $\mA$ and $\mB$ on the graph $\cG$, we say that ${\mA}\preceq {\mB}$ if $\mB-\mA$ is positive semi-definite. In this paper, the global objective function $F=\sum_{i\in V} f_i$ in the optimization problem \eqref{MainOptimizationProblem} are assumed to be smooth and strongly convex.

\begin{assumption}\label{globalobjectivefunction.assump}
There exist  positive constants $0<c<L<\infty$ and  positive definite matrices ${\mPhi}(\vx, \vy), \vx, \vy\in {\mathbb R}^{N}$, with geodesic-width $2m$  such that
\begin{equation}\label{GlobalHessianBounds.eq2}
\omega({\mPhi}(\vx, \vy))\le 2m\ \ {\rm and} \ \
  c {\mI}\preceq {\mPhi}(\vx, \vy)\preceq L {\mI},
\end{equation}
and
 \begin{equation}\label{GlobalHessianBounds.eq1}
\nabla F(\vx)-\nabla F(\vy)= {\mPhi}(\vx, \vy) (\vx-\vy)  \ \ {\rm for \ all} \ \vx, \vy\in {\mathbb R}^N.  \end{equation}
\end{assumption}

The requirement \eqref{GlobalHessianBounds.eq1}  can be considered as a strong version of its {\em strictly monotonicity} for the gradient $\nabla F$,
$$ (\vx-\vy)^T (\nabla F(\vx)-\nabla F(\vy)) \ge c  \|\vx-\vy\|^2 \ \ {\rm for \ all} \ \vx, \vy\in {\mathbb R}^N, $$
where $c$ is an absolute constant \cite{Zeidler1990, sun2014a}. On the other hand,  \Cref{globalobjectivefunction.assump} is satisfied  when the local objective functions $f_i, i\in V$, are twice differentiable and satisfy Assumption \ref{localobjectivefunction.assump}, and the global objective function $F$ satisfies the classical strict convexity condition
\begin{equation*}
c{\mI} \preceq \nabla^2 F(\vx)\preceq L{\mI}\ \ {\rm for \ all} \    \vx\in {\mathbb R}^N.
\end{equation*}
In particular, we have
$$   \nabla F(\vx)-
\nabla F({\vy}) =  \Big(\int_{0}^1 \nabla^2 F\big( t \vx+(1-t) {\vy}) dt\Big) (\vx-{\vy}) \ \ {\rm for \ all} \  {\vx, \vy}\in {\mathbb R}^N, $$
and
$$  \omega\big(\nabla^2 F(\vx)\big)\le 2m \ \ {\rm for \ all} \  {\vx}\in {\mathbb R}^N,
 $$
 where
$$
\frac{\partial^2 F(\vx)}{\partial x(i)\partial x(j)} = \sum_{k\in V}
  \frac{\partial^2 f_k(\vx) }{\partial x(i)\partial x(j)} %=   \sum_{k\in B(i,m)\cap B(j,m)} \frac{\partial^2 f_k(\vx) }{\partial x_i\partial x_j}
 = 0, \ \ \vx\in {\mathbb R}^N,$$
 hold for any $i,j\in V$ with $ \rho(i,j)>2m$,  since for any $k\in V$,   either $\frac{\partial f_k(\vx) }{\partial x(i)}=0$ or $\frac{\partial f_k(\vx) }{\partial x(j)}=0$ by \eqref{assumption2.characterization}.

\subsection{Fusion centers for distributed  and decentralized implementation}\label{fusioncenter.subsection}

In our distributed and decentralized implementation of the proposed  DAC algorithm, all processing of data storage, data exchange and numerical computation are conducted on fusion centers  located at some vertices of the graph $\cG$. Denote the location  set of fusion centers by $\Lambda$, a subset of the vertex set $V$ of the underlying graph $\cG$.   Associated with each fusion center $\lambda\in \Lambda$, we divide the whole set of vertices $V$ into a family of  {\em governing vertex sets}  $D_\lambda\subset V, \lambda\in \Lambda$ such that
\begin{equation}\label{governingvertices.def}
\cup_{\lambda\in \Lambda} D_\lambda=V\ \ {\rm and} \ \ D_\lambda\cap D_{\lambda'}=\emptyset \ \ {\rm for \ distinct} \ \lambda, \lambda'\in \Lambda.
\end{equation}
A common selection of the governing vertex sets is
the Voronoi diagram of $\cG$ with respect to $\Lambda$, which satisfies
\begin{equation*}
\{i\in V, \rho(i,\lambda)< \rho(i, \lambda') \ {\rm for \ all} \ \lambda'\ne \lambda\}\subset D_\lambda
\subset \{i\in V, \rho(i,\lambda)\le  \rho(i, \lambda') \ {\rm for \ all} \ \lambda'\ne \lambda\}, \ \lambda\in \Lambda.
\end{equation*}
In our setting, we do not  have any restriction on the sizes of the governing vertices $D_\lambda, \lambda\in \Lambda$, however it is more reasonable to assume that  $D_\lambda, \lambda\in \Lambda$, have similar sizes and are located in some ``neighborhood'' of  fusion centers, see our illustrative Example \ref{fusioncenter.example} below.

Denote the distance between two vertex subsets  $A, B$ of $V$ by $\rho(A, B)=\inf_{i\in A, j\in B} \rho(i,j)$.
In\ the proposed DAC algorithm,  we solve some local minimization problems on {\em extended $R$-neighbors} $D_{\lambda, R}$ of the governing vertex sets $D_\lambda, \lambda\in \Lambda$, which satisfy
\begin{equation}\label{extendedneighbor.def}
D_\lambda\subset D_{\lambda, R} \ \ {\rm and} \ \ \rho(D_\lambda, V\backslash D_{\lambda, R})>  R \ {\rm for \ all}\ \ \lambda \in \Lambda,
\end{equation}
where $R\ge 1$ is a positive number chosen later, see \eqref{deltaR.def}.  A simple choice of extended $R$-neighbors are the sets
$$D_{\lambda, R}=\cup_{i\in D_\lambda} B(i, R)=\{j\in V, \ \rho(j,i)\le R\ {\rm for \ some} \ i\in D_\lambda\},$$  of all $R$-neighboring vertices of $D_\lambda, \lambda\in \Lambda$.

Next we present an  example of the location set $\Lambda$ of fusion centers, the governing vertex sets $D_\lambda, \lambda\in \Lambda$, and their
extended $R$-neighbors  $D_{\lambda, R}, \lambda\in \Lambda$.
\begin{example} \label{fusioncenter.example}
We say that  $\Lambda\subset V$ is   a {\it maximal $R_0$-disjoint set} if
\begin{align*}
  B(i,R_0)\cap \Big(\bigcup_{\lambda\in \Lambda}B(\lambda,R_0)\Big)\neq \emptyset  \text{ for all } i\in V,
\end{align*}
and
\begin{align*}
  B(\lambda, R_0)\cap B(\lambda', R_0)=\emptyset  \  \text{ for  all distinct}\ \lambda, \lambda'\in \Lambda.
\end{align*}
Our illustrative example of  location/vertex set $\Lambda$ of fusion centers and the governing vertex sets $D_\lambda, \lambda\in \Lambda$ is a maximal $R_0$-disjoint set and the corresponding Voronoi diagram, see  Figure  \ref{FusionCenters.fig} where  $R_0=4$ and the underlying  graph  $\cG$ is  a geometric random graph with $N=2048$ vertices.

For  the above governing vertex sets $D_\lambda, \lambda\in \Lambda$, we have the following size estimate,
\begin{align*}
B(\lambda, R_0)\subseteq D_\lambda \subseteq B(\lambda, 2R_0) \ \ {\rm for \ all} \ \lambda\in \Lambda.
\end{align*}
For the case that  the sets of all $R$-neighboring vertices of $D_\lambda$ are chosen as extended $R$-neighbors
$D_{\lambda, R}=\cup_{i\in D_\lambda} B(i, R), \lambda\in \Lambda$, we obtain
 \begin{align*}
B(\lambda, R_0+R)\subseteq D_{\lambda, R} \subseteq B(\lambda, 2R_0+R) \ \ {\rm for \ all} \in \Lambda.
\end{align*}
\end{example}

We finish this section with a constructive approach, \Cref{GraphCovering.algorithm}, to construct a maximal $R_0$-disjoint set on  a connected undirected graph ${\cG}=(V, E)$ of order $N\ge 1$, and show that the total computational cost to implement \Cref{GraphCovering.algorithm} is  about $O(N^2)$. Here we say that two positive quantities $a$ and $b$ satisfy $b=O(a)$ if $b/a$ is bounded by an absolute constant $C$ independent on the order $N$
of the underlying graph $\cG$, which could be different at different occurrences and may depend on the radius $R_0$,
Beurling dimension $d({\mathcal G})$ and Beurling density $D_1({\mathcal G})$. Denote total number of steps used in the \Cref{GraphCovering.algorithm} by $M$ and the sets $U$ and $W$ at step $n$ by  $U_n$ and $W_n, 1\le n\le M$ respectively. Then one may verify by induction on $n\ge 1$ that
\begin{equation} \label{Wminclusion}
U_n\subset V\backslash W_n\subset \cup_{i\in U_n} B(i, R_0), \ 1\le n\le M,
\end{equation}
the sequence of cardinalities $\#W_n\in [0, N-1]$ of the sets $W_n, 1\le n\le M$, is  strictly decreasing,
\begin{equation}\label{Wmincreasingsequence}
\# W_{n+1}\le \#W_{n}-1, \ 1\le n\le M-1,\end{equation}
and  the sequence of  cardinalities $ \#U_n\in [1, N]$ of the sets $U_n, 1\le n\le M$ is increasing and has bounded increment,
\begin{equation}\label{Umincreasingsequence}
\# U_n \le \#U_{n+1}\le  \# U_n+ D_1({\cG})(R_0 + 1)^{d(\cG)}, \ 1\le n\le M-1,\end{equation}
where the last inequality  follows from \Cref{graphstructure.assumption}. By \eqref{Wminclusion} and \Cref{GraphCovering.algorithm}, the computational cost to find $j\in W_n$ and verify whether $B(j, R_0)\cap U_n=\emptyset$ for any given $j\in W_n$ is about $O(\#U_n), 1\le n\le M$. Therefore the total computational cost to implement \Cref{GraphCovering.algorithm} to find a maximal $R_0$-disjoint set is
 $$\sum_{n=1}^M O(\# U_n)=\sum_{n=1}^M O(n)= O(N^2), $$
where the first equality follows from  \eqref{Umincreasingsequence} and the second estimate hold as $M\le N$ by \eqref{Wmincreasingsequence}.

\begin{algorithm}[t] %[h!]
\caption{ Maximal $R_0$-disjoint subset algorithm}
\label{GraphCovering.algorithm}
\begin{algorithmic}  %[1]
\STATE{\bf Initialization}:  Pick one vertex $i\in V=\{1,2,...,N\}$,  and then set $\Lambda =\{i\}$, $U= B(i,R_0)$ and $W=V\backslash  B(i,R_0)$.
 \STATE{\bf for} $n=2,...,N$ \\
   \hspace{0.2in} Pick minimal $j\in W$  \\
   \hspace{0.2in} {\bf if} $B(j, R_0)\cap U= \emptyset$   \\
   \hspace{0.4in} $\Lambda\leftarrow \Lambda\cup\{j\}$;  $U\leftarrow U \cup B(j, R_0)$; $W\leftarrow W \backslash B(j, R_0)$\\
   \hspace{0.2in} {\bf else} \\
   \hspace{0.4in} $W\leftarrow W\backslash \{j\}$\\
   \hspace{0.2in} {\bf end}   \\
   \hspace{0.2in} {\bf if}   $\#W=0$\\
   \hspace{0.4in}    {\bf stop}\\
   \hspace{0.2in} {\bf end}   \\
\STATE{\bf end}  \\
\STATE{\bf Output:} $\Lambda$  \\
\end{algorithmic}\vspace{-.03in}\label{alg:fusioncenter}
\end{algorithm}

\section{ Divide-and-conquer algorithm and its distributed implementation}\label{DistributedOptimization.section}

Let  $G=(V,E)$ be a connected undirected graph,   $D_\lambda, \lambda\in \Lambda$, be a family of  governing vertex sets satisfying \eqref{governingvertices.def}, and $D_{\lambda, R}, \lambda\in \Lambda$ be their extended $R$-neighbors  satisfying \eqref{extendedneighbor.def}, see Section \ref{fusioncenter.subsection}. In this section, we introduce a novel DAC  algorithm \eqref{DACalgorithm.def}  to solve the  optimization problem \eqref{MainOptimizationProblem}, where we break down the global optimization problem \eqref{MainOptimizationProblem}
into  local minimization problems \eqref{DACalgorithm.defa} on overlapping extended $R$-neighbors  $D_{\lambda, R},\lambda\in \Lambda$, and then we combine the core part of solutions of the above local minimization problems to provide a better approximation to  the solution of the original minimization problem in each iteration, see \eqref{DACalgorithm.defb}. Due to  neighboring variable dependence of local objective functions $f_i, i\in V$,
we propose a scalable implementation of the DAC  algorithm \eqref{DACalgorithm.def} at fusion centers equipped with enough processing power and resources, see \Cref{Main.algorithm} for the implementation and \Cref{fusioncenter.assump} for the equipment requirement at fusion centers.

For a vector $\vx=(x(i))_{i\in V} \in \bR^{\# V}$ and a subset $W\subset V$,  we use $\chi_W: \bR^{\# V}\rightarrow \bR^{\# W}$ to denote the selection mapping $\chi_W \vx = (x(i))_{i\in W}$. Its adjoint mapping $\chi^*_W: \bR^{\# W}\rightarrow \bR^{\# V}$ is defined for a vector $\vu=(u(i))_{i\in W}$ as $\chi^*_W \vu \in \bR^{\# V}$ with the $i$-th component same as $u(i)$ when $i\in W$ and 0 otherwise. Moreover, we let $\mI_W = \chi^*_W \chi_W$ be the projection operator making the $i$-th block 0 for $i\notin W$.

In this paper, we propose the following iterative {\em divide-and-conquer algorithm}, DAC for abbreviation, to solve the  optimization problem \eqref{MainOptimizationProblem},
\begin{subequations}\label{DACalgorithm.def}
\begin{empheq}[left=\empheqlbrace]{align}
\vw_{\lambda}^n  &= \arg\min_{\vu \in \bR^{\# D_{\lambda, R}}}
 F(\chi^*_{D_{\lambda, R}}\vu+ \mI_{V\backslash D_{\lambda, R}} \vx^n), \ \lambda\in \Lambda, \label{DACalgorithm.defa} \\
 \vx^{n+1}  &=  \sum_{\lambda\in \Lambda}\mI_{D_{\lambda}}\chi^*_{D_{\lambda, R}} \vw_{\lambda}^n, \
\ n\ge 0,  \label{DACalgorithm.defb}
\end{empheq}
\end{subequations}
 where  an initial $\vx^0$ is  arbitrarily or randomly chosen. The iteration step in the DAC algorithm solves a family of the local minimization  problems \eqref{DACalgorithm.defa} on the overlapping extended $R$-neighbors ${D}_{\lambda,R}, \lambda\in \Lambda$, and combines the core part $D_\lambda$ of  solutions  $\vw_\lambda^n, \lambda\in \Lambda$ of those local minimization problems to provide a better approximation to the solution of the original minimization problem when  the radius parameter $R\ge 0$ is appropriately chosen, see \Cref{ConvergenceTheorem}.

For  a subset $W\subset V$ and $i\in V$,   we use $\rho(i, W)=\inf_{j\in W} \rho(i,j)$ to denote the distance between the vertex $i$ and the set $W$. Let $m\ge 1$ be the neighboring radius of the local objective functions $f_i, i\in V$, and define $l$-neighbors $D_{\lambda, R, l}$ of $D_{\lambda, R}, \lambda\in \Lambda$ by
\begin{equation*}
D_{\lambda, R, l}=\{i\in V,\ \rho(i, D_{\lambda, R})\le l\}, \ \lambda\in \Lambda.
\end{equation*}
For any $\vy\in {\mathbb R}^N$ and $\lambda\in \Lambda$, we obtain from  \Cref{localobjectivefunction.assump} that
\begin{eqnarray}\label{localminimization.def0}
 &&\arg\min_{\vu\in \bR^{\# D_{\lambda, R}}}
 F(\chi^*_{D_{\lambda, R}}\vu+ \mI_{V\backslash D_{\lambda, R}}\vy) \nonumber\\
    & \hskip-0.08in = &  \arg\min_{\vu\in \bR^{\# D_{\lambda, R}}} \Big(\sum_{i\in  D_{\lambda, R, m}}+ \sum_{i\not\in D_{\lambda, R, m}}\Big)
 f_i(\chi^*_{D_{\lambda, R}}\vu+ \mI_{V\backslash D_{\lambda, R}}\vy)\nonumber\\
    & \hskip-0.08in = &  \arg\min_{\vu\in \bR^{\# D_{\lambda, R}}} \sum_{i\in  D_{\lambda, R, m}}
 f_i(\chi^*_{D_{\lambda, R}}\vu+ \mI_{V\backslash D_{\lambda, R}}\vy)\nonumber\\
     & \hskip-0.08in = &  \arg\min_{\vu\in \bR^{\# D_{\lambda, R}}} \sum_{i\in  D_{\lambda, R, m}}
 f_i(\chi^*_{D_{\lambda, R}}\vu+ \mI_{D_{\lambda, R, 2m}\backslash D_{\lambda, R}}\vy)
\end{eqnarray}
Based on the above observation,  the iterative DAC algorithm \eqref{DACalgorithm.def}  can be implemented at fusion centers, see Algorithm \ref{Main.algorithm} for the implementation at  each fusion center $\lambda\in \Lambda$ where
\begin{equation}\label{neighboringfusioncenter.def}
\Lambda_{\lambda,  R,  2m}^{\rm out}=\{\lambda'\in \Lambda, D_{\lambda}\cap D_{\lambda', R, 2m}\ne \emptyset\} \ {\rm and} \
\Lambda_{\lambda, R,  2m}^{\rm in}=\{\lambda'\in \Lambda, D_{\lambda'}\cap D_{\lambda, R, 2m}\ne \emptyset\}, \lambda \in \Lambda.
\end{equation}
 We  consider any fusion center $\lambda'\in \Lambda_{\lambda,  R,  2m}^{\rm out}$ and $\Lambda_{\lambda,  R,  2m}^{\rm in}$ as an {\em out-neighbor} and {\em in-neighbor} of the fusion center $\lambda\in \Lambda$ respectively.
 For  $\lambda, \lambda'\in \Lambda$, one may verify that  $\lambda'\in \Lambda_{\lambda, R,2m}^{\rm out}$ if and only if $\lambda\in \Lambda_{\lambda', R,2m}^{\rm in}$
 and hence in Algorithm \eqref{DACalgorithm.def} the data vector  transmitted from a fusion center  to its out-neighbors will be received.

\begin{algorithm}[t] %[h!]  %[h]
\caption{Implementation of the  DAC algorithm  \eqref{DACalgorithm.def} at a fusion center $\lambda\in \Lambda$.}
\label{Main.algorithm}
\begin{algorithmic}  %[1]
\STATE{\bf Initialization}:  Maximal iteration number $T$, vertex sets $D_\lambda, D_{\lambda, R}$ and  $D_{q, R, 2m}$,
local objective functions $f_i, i\in D_{\lambda, R, m}$,  neighboring fusion  sets
$\Lambda_{\lambda, R, 2m}^{\rm out}$ and
$\Lambda_{\lambda, R, 2m}^{\rm in}$ in \eqref{neighboringfusioncenter.def},
 and initial guess $\chi_{D_{\lambda, R, 2m}}{ \bf x}^0$, i.e., its components $x_0(i), i\in D_{\lambda, R, 2m}$.
\STATE{\bf Iteration}:
\begin{itemize}
\item[{}] {\bf for} \ $n=0, 1, \ldots, T$
\item[{}]  Solve the local minimization problem
\begin{align*}
\vw_{\lambda}^n  = \arg\min_{\vu \in \bR^{\# D_{\lambda, R}}}
\sum_{i\in D_{\lambda, R, m}} f_{i}(\chi^*_{D_{\lambda, R}}\vu+\mI_{D_{\lambda, R, 2m}\backslash D_{\lambda, R}}\vx^{n})
\end{align*}

\item[{}] Send $\chi_{D_{\lambda}}\chi^*_{D_{\lambda, R}} \vw_{\lambda}^n$ to neighboring fusion centers $\lambda'\in \Lambda_{\lambda, R, 2m}^{\rm out}$;

\item[{}] Receive  $\chi_{D_{\lambda'}} \chi^*_{D_{\lambda', R}} \vw_{\lambda'}^n$ from neighboring fusion centers $\lambda'\in \Lambda_{\lambda, R, 2m}^{\rm in}$;

\item[{}] Evaluate $\chi_{D_{\lambda, R, 2m} \backslash D_{\lambda, R} } \vx^{n+1}  = %\chi_{D_{\lambda}} \vw_{\lambda}^n+
\chi_{D_{\lambda, R, 2m}\backslash D_{\lambda, R} } \sum_{\lambda'\in \Lambda_{\lambda, R,  2m}^{\rm in}}
\mI_{D_{\lambda'}}  \chi^*_{D_{\lambda',R}}\vw_{\lambda'}^n$.

\item[{}] {\bf end}
\end{itemize}

\STATE{\bf Output}: $\chi_{D_\lambda}\vx^{T+1}:=\chi_{D_{\lambda}} \chi^*_{D_{\lambda, R}}\vw_\lambda^{T}$
% $\chi_{D_{\lambda, R, 2m}} \vx^{M+1}$.
\end{algorithmic}\vspace{-.03in}
\end{algorithm}

For the implementation of Algorithm \ref{Main.algorithm}  at  fusion centers, it is required that each fusion center is  equipped with  enough memory for data storage, proper communication bandwidth to exchange data with its neighboring fusion centers, and high computing power to solve local minimization problems.

\begin{assumption}\label{fusioncenter.assump}
(i)\  Each fusion center $\lambda\in \Lambda$ can  store vertex sets $D_\lambda, D_{\lambda, R},  D_{\lambda, R,  m}$ and $D_{\lambda, R, 2m}$, neighboring fusion  sets $\Lambda_{\lambda, R, 2m}^{\rm out}$ and $\Lambda_{\lambda, R, 2m}^{\rm in}$, and the vectors  $\chi_{D_{\lambda, R, 2m} } \vx^{n}$  and %$\chi_{D_{\lambda, R}}
$\vw_\lambda^n$ at each iteration, and reserve enough memory used for storing local objective functions $f_i, i\in D_{\lambda, R, m}$ and solving the local minimization problem \eqref{localminimization.def0} at each iteration;

(ii)\ each fusion center $\lambda\in \Lambda$ has computing facility to solve the local minimization problem \eqref{localminimization.def0} quickly; and

(iii)\ each fusion center $\lambda\in \Lambda$ can send data $\chi_{D_\lambda} \chi^*_{D_\lambda, R} \vw_\lambda^n$ to fusion centers $\lambda' \in  \Lambda_{\lambda, R, 2m}^{\rm out}$ and receive data $\chi_{D_{\lambda'}} \chi^*_{D_{\lambda', R}} \vw_{\lambda'}^n$ from fusion centers $\lambda' \in  \Lambda_{\lambda, R, 2m}^{\rm in}$ at each iteration.
\end{assumption}

We finish this section with a remark on the above assumption when location set $\Lambda$ of fusion centers, the governing vertex sets $D_\lambda, \lambda\in \Lambda$
and the extended $R$-neighbors $D_{\lambda, R}, \lambda\in \Lambda$
are the maximal $R_0$-disjoint set,  its corresponding Voronoi diagram and
the set of all $R$-neighboring vertices of the Voronoi diagram
 respectively, see \Cref{fusioncenter.example}.
In this case, we have
$$D_\lambda\subset B(\lambda, 2R_0),\  D_{\lambda, R}\subset B(\lambda, 2R_0+R),  \ \
D_{\lambda, R, l}\subset B(\lambda, 2R_0+R+l)\ {\rm for}\ l=m, 2m,
$$
and
$$ \Lambda_{\lambda, R, 2m}^{\rm out} \subset \{\lambda'\in \Lambda, B(\lambda, 2R_0)\cap B(\lambda', 2R_0+R+2m)\ne \emptyset\}
\subset  B(\lambda, 4R_0+R+2m)\cap \Lambda , \ \lambda\in \Lambda,$$
$$ \Lambda_{\lambda, R, 2m}^{\rm in}\subset \{\lambda'\in \Lambda, B(\lambda', 2R_0)\cap B(\lambda, 2R_0+R+2m)\ne \emptyset\}
\subset  B(\lambda, 4R_0+R+2m)\cap \Lambda , \ \lambda\in \Lambda,$$
see Figure \ref{FusionCenters.fig}. This together with \Cref{graphstructure.assumption}
implies that
the local minimization problem of the form \eqref{localminimization.def0} is of size at most $D_1(\cG) ( 2R_0+R+1)^{d(\cG)}$,
each fusion center has  at most $D_1(\cG) ( 4R_0+R+2m+1)^{d(\cG)}$ neighboring fusion centers (usually it is much smaller), and in addition to memory requirement
to solve local minimization problem \eqref{localminimization.def0} of dimension at most $D_1(\cG) ( 2R_0+R+1)^{d(\cG)}$, each fusion center has memory of size
$D_1(\cG) ( 4R_0+R+2m+1)^{d(\cG)}$ to store some vertex sets and vectors.
From the convergence result in Theorem \ref{ConvergenceTheorem}, the parameter $R$ in
the DAC algorithm  can be chosen to depend only on the constants in \Cref{graphstructure.assumption}, \Cref{localobjectivefunction.assump} and \Cref{globalobjectivefunction.assump}, see \eqref{deltaR.def}.
Therefore  to meet  the requirements in \Cref{fusioncenter.assump},
in the above setting we could equip  data storage, communication devices and computing facility at each fusion center  independently on the order  $N$
of the underlying graph $\cG$,  and hence
 the distributed implementation of Algorithm \ref{Main.algorithm}  at  fusion centers is scalable.

\section{Exponential convergence of the divide-and-conquer algorithm}\label{convergence.section}

Let $\ell^p, 1\le p\le \infty$, be the space of all  $p$-summable sequences on the graph $\cG$ with its standard norm denoted by $\|\cdot\|_p$.
In  \Cref{ConvergenceTheorem} of  this section, we show that $\vx^n, n\ge 1$, in the proposed
iterative
DAC  algorithm
 \eqref{DACalgorithm.def} converges
  exponentially to  the solution $\hat \vx$
 of the  optimization problem \eqref{MainOptimizationProblem} in $\ell^p, 1\le p\le \infty$, when the parameter $R\ge 1$ in the extended neighbors of governing vertex sets
is chosen so that % $R\ge 2m+1$ and
%$$\delta_M=\frac{L^2 D_1(\cG)(2m+1)^{d(\cG)}(3M+2m+1)^{d(\cG)}}{c^2}e^{-\frac{M\theta}{2m}+\theta}<1,$$
\begin{equation}\label{deltaR.def}
\delta_R:=\frac{D_1(\cG) d(\cG)! L^2 (2m)^{d(\cG)} }{c(L-c)  |\ln (1-c/L)|^{d(\cG)}}   (R+2)^{d({\cG})}\Big(1-\frac{c}{L}\Big)^{R/(2m)}
%\frac{L ( D_1(\cG))^{2}  (2m+1)^{2d(\cG)+1 }  d(\cG)! }{c (\ln (L/(L-c)))^{d(\cG)-1}  }  \Big(1-\frac{c}{L}\Big)^{(R-2m-1)/(2m)} (R+1)^{d(\cG)}
%\frac{L^2  (D_1(\cG))^2  (2m+1)^{d(\cG) } }{(L-c)c}
% \sum_{k=R+1}^\infty   (k+1)^{d(\cG)} \Big(1-\frac{c}{L}\Big)^{k/(2m)}
 <1.\end{equation}
  In many practical applications, it could be numerically expensive to solve the local minimization problem
\eqref{localminimization.def0}  exactly.
In Theorem \ref{inexactConvergenceTheorem}, we provide an error estimate when  the local minimization problems
in the DAC algorithm are solved up to certain bounded accuracy.

\begin{theorem}\label{ConvergenceTheorem}
Let the underlying graph  $\cG=(V, E)$, the local objective functions $f_i, i\in V$
and the global objective function $F$ of the
  optimization problem \eqref{MainOptimizationProblem}
  satisfy Assumptions \ref{graphstructure.assumption}, \ref{localobjectivefunction.assump} and \ref{globalobjectivefunction.assump} respectively,
  $\hat \vx={\rm arg} \min_{\vx\in {\mathbb R}^N}F(\vx)$ be the unique solution   of the global minimization problem  \eqref{MainOptimizationProblem},
  and $\{\vx^n, n\ge 0\}$ be the sequence generated in the iterative DAC algorithm \eqref{DACalgorithm.def}.
If the parameter $R\ge 1$ in \eqref{DACalgorithm.def} is chosen to satisfy \eqref{deltaR.def},
 then $\{\vx^n, n\ge 0\}$ converges to
 $\hat \vx\in {\mathbb R}^N$ exponentially in $\ell^p, 1\le p\le \infty$, with convergence rate $\delta_R$,
\begin{equation}\label{ConvergenceTheorem.eq4}
\|\vx^n-\hat \vx\|_{p}\le \left(\delta_{R}\right)^n \|\vx^0-\hat\vx\|_{p},\  n\ge 0. % + \sum_{j=1}^n \epsilon_j \delta_R^{n-j}, \ n\ge 0.
\end{equation}
 \end{theorem}

For a matrix $\mA=(a(i,j))_{i,j\in V}$ on the graph $\cG=(V, E)$, define its entrywise bound by
 $\|{\mA}\|_{\infty}=\sup_{i,j\in V}|a(i,j)|$, its operator norm  on $\ell^p, 1\le p\le \infty$, by
$$\|{\mA}\|_{{\mathcal B}_p}=\sup_{\|\vx\|_p=1} \|{\mA}\vx\|_p,$$
and its Schur norm
 by
 $$\|{\mA}\|_{\mathcal{S}}=\max\Big(\sup_{i\in V} \sum_{j\in V} |a(i,j)|, \sup_{j\in V} \sum_{i\in V} |a(i,j)|\Big)=\max_{1\le p\le\infty}
 \|{\mA}\|_{{\mathcal B}_p}.$$
 As shown in \cite[Prop. III.3]{Jiang19}, a bounded matrix $\mA$ with limited geodesic-width $\omega({\mA})$ is a bounded operator on $\ell^p, 1\le p\le \infty$.
\begin{equation}\label{NormInequality}
\|{\mA}\|_{\infty}\le \|{\mA}\|_{\mathcal{B}_p}\le \|{\mA}\|_{\mathcal{S}}\le D_1(\cG)(\omega({\mA})+1)^{d(\cG)}\|{\mA}\|_{\infty}.
\end{equation}
The crucial step in the proof of Theorem \ref{ConvergenceTheorem} is to find matrices
${\bf H}_n, n\ge 0$, on the graph $\cG$ such that
\begin{equation}\label{maintheorem.pf-1}
\vx^{n+1}-\hat \vx= {\bf H}_n (\vx^n-\hat \vx)\ \ {\rm and} \ \ \|{\bf H}_n\|_{\mathcal S}\le \delta_R, \ n\ge 0,\end{equation}
 see
\eqref{ConvergenceTheorem.pfeq7} and \eqref{ConvergenceTheorem.pfeq17}.
The detailed argument of  Theorem \ref{ConvergenceTheorem}  will be given in  \Cref{convergencetheorem.proofsection}.

\smallskip
It is direct to observe from
 \Cref{globalobjectivefunction.assump} that the local optimizer $\vw^n_\lambda$  in
\eqref{DACalgorithm.defa}
satisfies
  \begin{equation}\label{ConvergenceTheorem.pfeq0}
 \chi_{D_{\lambda, R}} \nabla F(\chi^*_{D_{\lambda,R}}\vw_{\lambda}^n+\mI_{V\backslash D_{\lambda,R}} \vx^{n})={\bf 0}, \ \lambda\in \Lambda.
\end{equation}
This motivates us to consider the following inexact DAC algorithm  starting
from an initial $\tilde \vx^0$ arbitrarily or randomly chosen,
and solving a family of the local minimization  problems
with bounded accuracy $\epsilon_n, n\ge 0$,
\begin{subequations}\label{inexactDACalgorithm.def}
\begin{equation}\label{inexactDACalgorithm.defa}
\tilde \vw_{\lambda}^n \ \ \text{is selected such that}  \ \
\| \chi_{D_{\lambda, R}}\nabla F(\chi^*_{D_{\lambda, R}}\tilde \vw_{\lambda}^n + \mI_{V\backslash D_{\lambda, R}}\tilde \vx^n)\|_\infty \le \epsilon_n,\ \lambda\in \Lambda,
\end{equation}
%\vskip-0.18in
and  combining the core part of the above inexact solutions
\begin{equation}\label{inexactDACalgorithm.defb}
\tilde \vx^{n+1}  =  \sum_{\lambda\in \Lambda} \mI_{D_{\lambda}}\chi^*_{D_{\lambda, R}} \tilde \vw_{\lambda}^n, \
\ n\ge 0,
\end{equation}
\end{subequations}
to  provide a novel approximation.
Shown in the following theorem is the error estimate $\|\tilde \vx^n-\hat \vx\|_\infty, n\ge 0$.

\begin{theorem}\label{inexactConvergenceTheorem}
Let the underlying graph  $\cG=(V, E)$,  the global objective function $F$,
the optimal point $\hat \vx$,
the parameters $R$, and the convergence rate  $\delta_R$ be as in Theorem \ref{ConvergenceTheorem},
and  $\tilde \vx^n, n\ge 1$, be the solutions of the inexact DAC algorithm \eqref{inexactDACalgorithm.def} with bounded accuracy $\epsilon_n, n\ge 0$.
Then  %there exist a positive constant $C$ such that
\begin{equation}\label{inexactConvergenceTheorem.eq1}
\|\tilde \vx^n-\hat\vx\|_{\infty}\le (\delta_{R})^n \|\vx^0-\hat\vx\|_{p}
+\frac{L-c}{L^2} \sum_{m=0}^{n-1} (\delta_{R})^{n-m} \epsilon_{m},\  n\ge 1. % + \sum_{j=1}^n \epsilon_j \delta_R^{n-j}, \ n\ge 0.
\end{equation}
 \end{theorem}

 The crucial step in the proof of Theorem \ref{inexactConvergenceTheorem}
 is to  find matrices $\tilde {\bf H}_n, n\ge 0$,  such that
\begin{equation*}
\|\tilde {\bf H}_n\|_{\mathcal S}\le \delta_R\ \  {\rm and}\ \
\|\vx^{n+1}-\hat \vx-\tilde {\bf H}_n (\vx^n-\hat \vx)\|_\infty\le  \frac{L-c}{L^2} \delta_R \epsilon_n, \ n\ge 0,\end{equation*}
similar to
\eqref{maintheorem.pf-1} where $\epsilon_n=0$ for all $n\ge 0$. The detailed argument of Theorem
\ref{inexactConvergenceTheorem} is given in \Cref{inexactConvergenceTheorem.proofsection}.

\smallskip
By Theorem \ref{inexactConvergenceTheorem}, we conclude that
the solutions of the inexact DAC algorithm \eqref{inexactDACalgorithm.def} converges to
the true  optimal point when the accuracy bounds $\epsilon_n, n\ge 0$, have zero limit.

\begin{corollary}
Let the underlying graph  $\cG=(V, E)$,  the global objective function $F$,
the optimal point $\hat \vx$,
the parameters $R$, and the convergence rate  $\delta_R$,
and the inexact solutions
$\tilde \vx^n$  and accuracy  bounds $\epsilon_n, n\ge 0$, be as in Theorem \ref{inexactConvergenceTheorem}.
If $\lim_{n\to \infty} \epsilon_n=0$, then
$\lim_{n\to \infty} \tilde \vx^n= \hat \vx$.
 \end{corollary}

\section{Numerical Examples}\label{NumericExamples}

In this section, we demonstrate the performance of the iterative DAC for the least squares minimization problem  and the LASSO model, and  we make a comparison with the performance of some popular decentralized optimization methods, including decentralized gradient descent (DGD) \cite{Nedic2009,Matei2011,Yuan2016a}, Diffusion \cite{Cattivelli2010,Cattivelli2007}, exact first-order algorithm
(EXTRA) \cite{Shi15}, proximal gradient exact first-order algorithm (PG-EXTRA) \cite{Shi2015b}, and  network-independent step-size algorithm
(NIDS) \cite{Li2019a}. The numerical experiments show that, comparing with DGD, Diffusion, EXTRA, PD-EXTRA and NIDS,
 the proposed DAC method has superior performance for the
 least squares problem with/without $\ell^1$-penalty and has much faster  convergence. Moreover, the computational cost of the proposed DAC is almost linear with respect to the graph size, and hence  it has a great potential in scalability to work with extremely large networks.

In all the experiments below,  the underlying graph of our distributed optimization problems
is the  random geometric graph ${\mathcal G}_N=(V_N, E_N)$ of order $N\ge 256$,  where vertices are randomly deployed in the unit square $[0, 1]^2$ and  an edge between two vertices exist if their Euclidean distance is not larger than $\sqrt{2/N}$, see Figure \ref{FusionCenters.fig}.
On the random geometric  graph $\cG_N$, we denote its
 adjacent matrix, degree matrix and
 symmetric normalized Laplacian matrix by $\mA, {\bf D}$ and $\mL^{\rm{sym}}={\bf I}-{\bf D}^{-1/2} {\bf A}{\bf D}^{-1/2}$ respectively.

 All the numerical experiments are implemented using Python 3.8 on a computer server with Intel(R) Xeon(R) Gold 6148 CPU  2.4GHz and 32G memory.

\subsection{Least Squares Problem}\label{LeastSquare.Problem}

Consider the following least square problem:
\begin{equation}\label{eq:least_square}
\hat{\vx}=\arg\min_{\vx\in \bR^N} \left\{ F(\vx)=\frac{1}{2} \|\mH \vx-\vb\|_2^2\right\}
\end{equation}
where $\mH = \mI + 5 \mL^{\rm{sym}}$ and $\vb=(b(i))_{i\in V}\in \bR^N$ is randomly generated from normal distribution with mean $0$ and variance $1$.
Write ${\bf H}=(H(i,j))_{i,j\in V_N}$ and define
\begin{align*}
f_i(\vx) = \frac{1}{2} \biggl(\sum_{j\in V_N}H(i,j)x(j) - b(i) \biggr)^2, \quad i\in V_N
\end{align*}
for $\vx=(x(i))_{i\in V_N}$.
As the symmetric normalized Laplacian matrix
${\bf L}^{\rm sym}$ has geodesic-width one and satisfies ${\bf 0} \preceq {\bf L}^{\rm sym} \preceq 2{\bf I}$,
Assumptions \ref{localobjectivefunction.assump} and \ref{globalobjectivefunction.assump} are satisfied with
\begin{equation*} m=1,\  c=1,\  L=121 \ {\rm and} \ {\bf \Phi}(\vx, \vy)={\bf H}^2 \ {\rm for} \ \vx, \vy\in {\mathbb R}^N.
\end{equation*}
To implement the  DAC algorithm \eqref{DACalgorithm.def}, we   start from applying \Cref{alg:fusioncenter} to find  a maximal $R_0$-disjoint set with $R_0=1$,
use the maximal $R_0$-disjoint set and its corresponding Voronoi diagram as the location set  $\Lambda$ of fusion centers and  the family of  governing vertex sets $D_\lambda, \lambda\in \Lambda$.  Next we select $R=3$ and use the set of all $R$-neighbors of $D_{\lambda}$ as the extended $R$-neighbors  $D_{\lambda, R}, \lambda\in \Lambda$.
Having the fusion centers selected and extended $R$-neighbors ready, we then follow \Cref{Main.algorithm} to implement the  DAC algorithm \eqref{DACalgorithm.def} with zero  initial $\vx^0={\bf 0}$  and use ${\|\vx^{n+1} - \vx^n\|_2}/{\|\vx^n\|_2}\le 10^{-14}$ as the stopping criterion.

We will compare the performance of our proposed algorithm with DGD,  Diffusion, and EXTRA. In this regard, we will use the partition $\{D_\lambda: \Lambda\in \Lambda\}$ as the nodes in DGD. i.e.,  each block $D_\lambda$ of vertices would be treated as a single node in DGD and two nodes $D_\lambda$ and $D_{\lambda'}$ are connected if the two fusion centers $\lambda$ and $\lambda'$ are within distance $2(R+R_0)=8$, i.e.,  $\rho(\lambda, \lambda')\leq 8$. We define the objective function at each fusion center:
\begin{align*}
F_\lambda(\vx)=\sum_{i\in D_\lambda}f_i(\vx) = \frac{1}{2}\|\mH_{D_\lambda}\vx - \vb_{D_\lambda}\|^2 , \quad \lambda\in \Lambda,
\end{align*}
where $\mH_{D_\lambda} = (H(i,j))_{i\in D_\lambda, j\in V_N}$ and $\vb_{D_\lambda}=(b_i)_{i\in D_\lambda}$.
Denote   the stacked local copies $\vx_\lambda$  and  the stacked local gradients $\nabla F_\lambda(\vx_\lambda)$
held at each fusion center $\lambda\in \Lambda$ by
$\mX$ and $\nabla F(\mX)$ respectively, i.e.,
\begin{align*}
\mX= \begin{pmatrix}
(\vx_1)^T\\
\vdots \\
(\vx_{\# \Lambda})^T
\end{pmatrix}
\quad
\mbox{and}
\quad
\nabla F(\mX)=
\begin{pmatrix}
\nabla^T F_1(\vx_1)\\
\vdots\\
\nabla^T F_{\# \Lambda}(\vx_{\# \Lambda})
\end{pmatrix}.
\end{align*}
 In the DGD, Diffusion and EXTRA,  the following iteration schemes
\begin{align*}
&\mbox{DGD}  &\mX^{n+1} = &\mW \mX^n - \alpha \nabla F(\mX^n),\\
&\mbox{Diffusion} &\mX^{n+1} = &\mW (\mX^n - \beta \nabla F(\mX^n)),\\
&\mbox{EXTRA}  &\mX^{n+1}  = &(\mI + \mW)\mX^{n} - \frac{\mI + \mW}{2} \mX^{n-1} - \gamma (\nabla F(\mX^{n}) - \nabla F(\mX^{n-1})),
\end{align*}
are used,
where the step sizes are set to be $\alpha={0.99}{(\max_{\lambda\in \Lambda} \|\mH_{D_\lambda}\|^2)^{-1}}$ and $\beta=\gamma=2\alpha$.
There are several different choices of the mixing matrix $\mW$ \cite{Yuan2016a,Sayed2014,Xiao2004,Boyd2004,Xiao2007}. In our simulations, we
use the Metropolis mixing matrix \cite{Boyd2004,Xiao2007} with entries  $w_{\lambda,\lambda'}, \lambda, \lambda'\in \Lambda$ defined by
\begin{align}\label{eq:mixingmatrix}
w_{\lambda,\lambda'}=\begin{cases}
(\max\{\# N_\lambda, \# N_{\lambda'}\} + 0.1)^{-1} & {\rm if} \ \lambda, \lambda' \mbox{ are connected}\\
1-\sum_{\tilde \lambda \in N_\lambda}W_{\lambda,\tilde \lambda} & {\rm if} \ \lambda'=\lambda\\
0 & \mbox{otherwise}
\end{cases}
\end{align}
as  the mixing matrix ${\bf W}$, where  $N_\lambda$ is the set of all nodes connected to the vertex $\lambda\in \Lambda$. %denotes the neighbors of $i$, that is, all the nodes connected to node $i$.

We test our proposed DAC method, DGD, Diffusion, and EXTRA on four geometric random graphs with size 256, 512, 1024, and 2048 respectively and plot in \Cref{fig:ls} the average logarithm error $\log_{10}\|\vx^n - \hat{\vx}\|$ with respect to different running time (seconds) over 100 random selections of the observation vector $\vb$.

We observe from \Cref{fig:ls} that the proposed DAC method has superior performance to all the other three, it takes much less computational time to reach the machine accuracy and the corresponding computational time is close to a linear dependence on the order of the graph. This strongly suggests that the proposed DAC method has a strong scalability and the great potential to apply distributed algorithms on networks of extremely large size.

\begin{figure}[t] %[!h]
\begin{subfigure}[t]{0.48\textwidth}
\includegraphics[width = \textwidth]{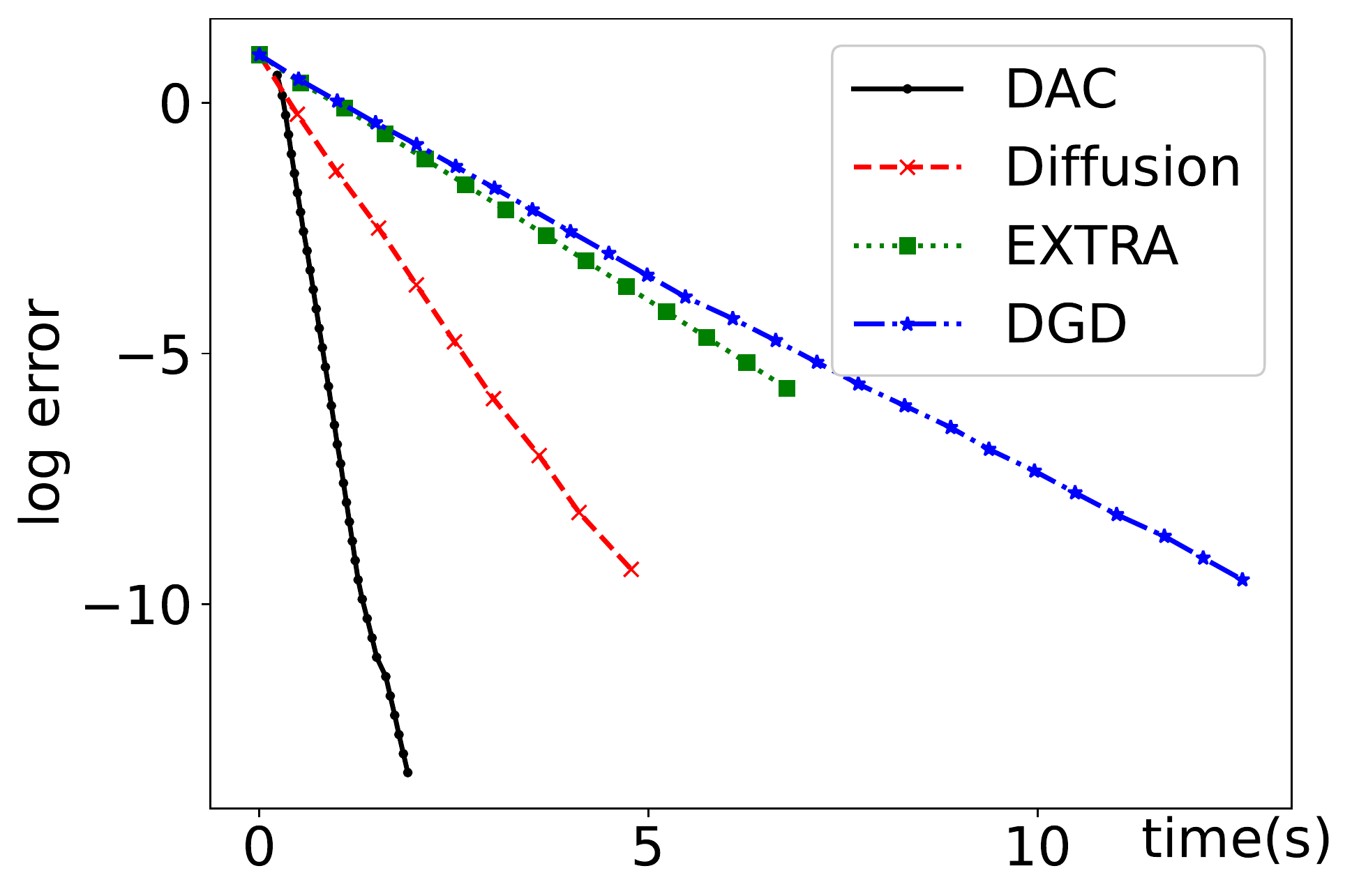}
\caption{$N=256$}
\end{subfigure}
\begin{subfigure}[t]{0.48\textwidth}
\includegraphics[width = \textwidth]{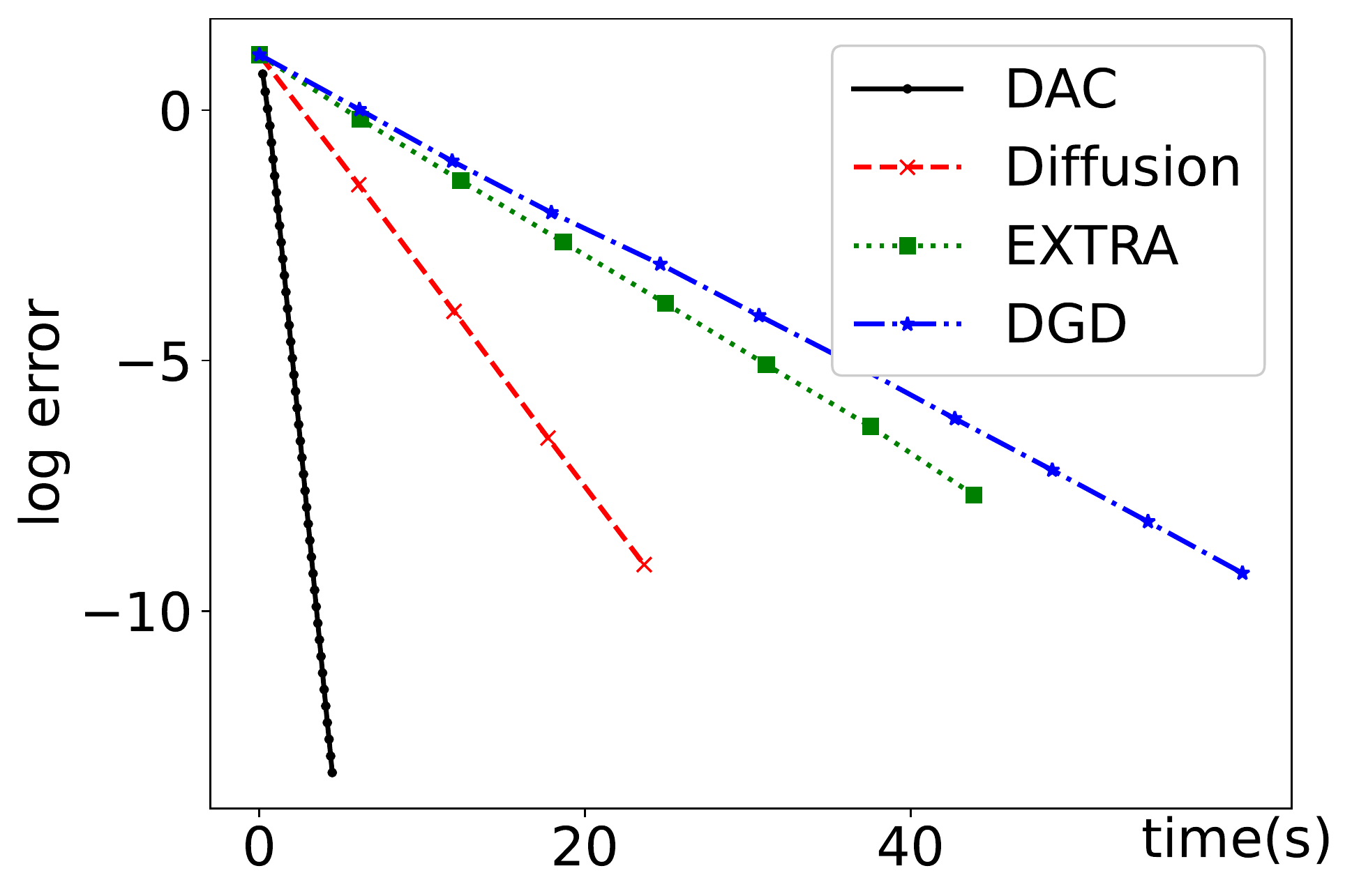}
\caption{$N=512$}
\end{subfigure}\\
\begin{subfigure}[t]{0.48\textwidth}
\includegraphics[width = \textwidth]{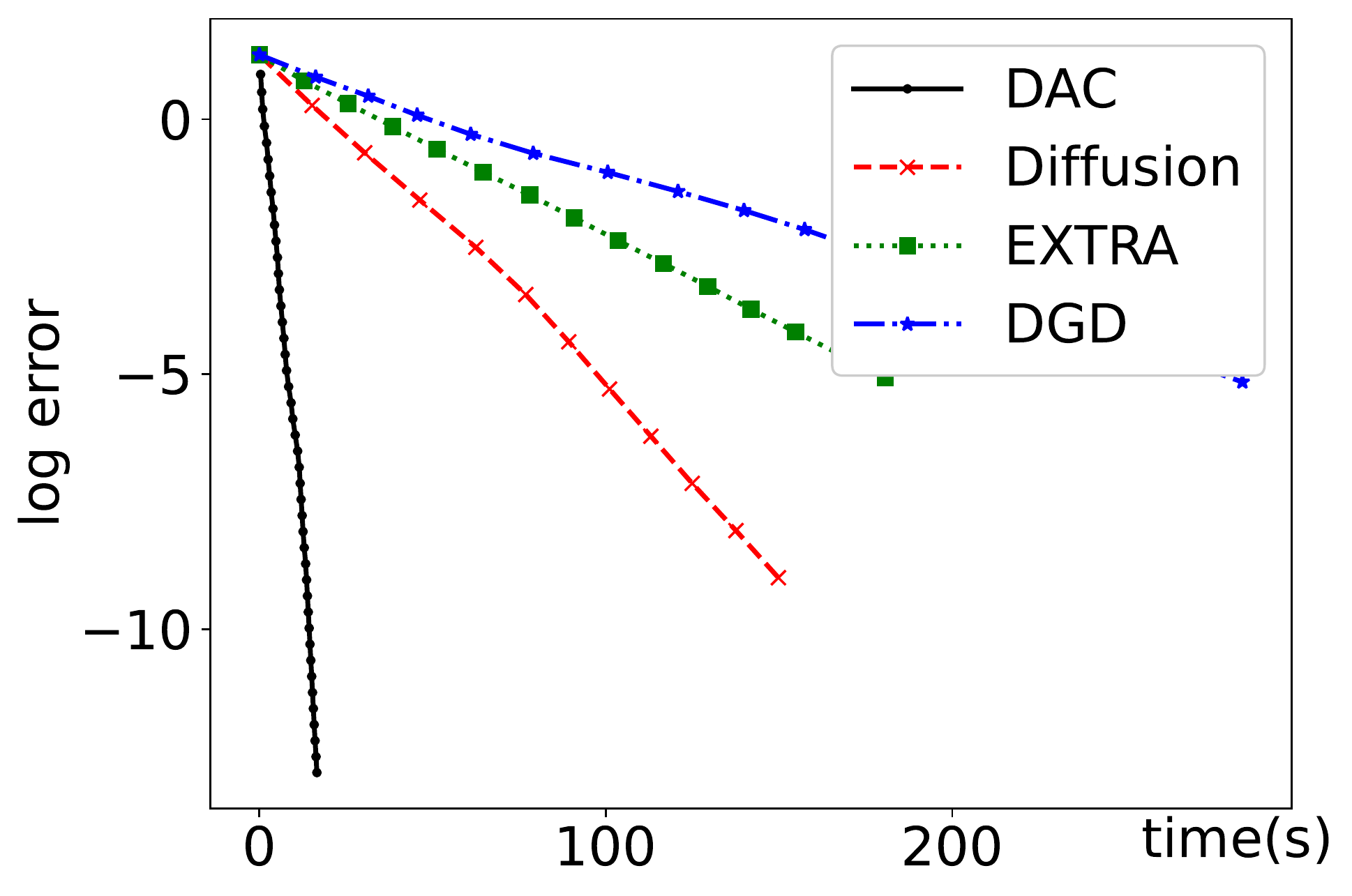}
\caption{$N=1024$}
\end{subfigure}
\begin{subfigure}[t]{0.48\textwidth}
\includegraphics[width = \textwidth]{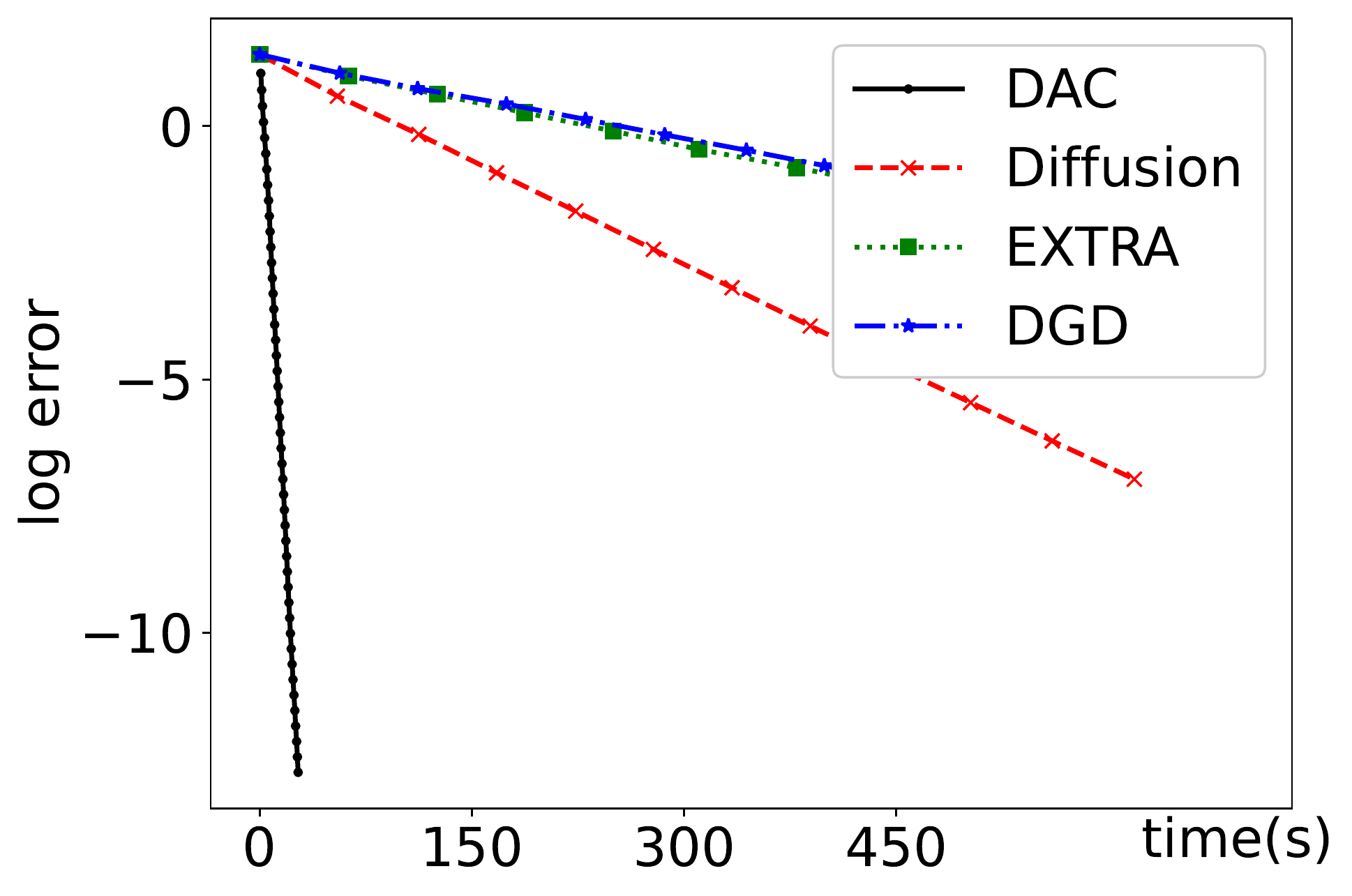}
\caption{$N=2048$}
\end{subfigure}
\caption{The logarithm error over running time (seconds) for  least square problems \eqref{eq:least_square} on random geometric graphs of size $N$}
\label{fig:ls}
\end{figure}

\subsection{LASSO}
In this subsection, we consider solving the following least squares problem with $\ell^1$ penalty,
\begin{equation}\label{lasso}
\hat\vx=\arg\min_{\vx\in {\mathbb R}^N} \Big\{ F(\vx)=\frac{1}{2} \|{\bf H}\vx-{\bf b}\|_2^2+\mu \|\vx\|_1\Big\}
\end{equation}
on the random geometric graph ${\mathcal G}_N$,
where $\mH, \vb$ are the same as those in \cref{eq:least_square} and $\mu=10$.
In the implementation of the proposed DAC algorithm, we use  the same $\Lambda$, $D_\Lambda$, $D_{\lambda,R}$ and the stopping criterion as the ones  in  Section \ref{LeastSquare.Problem}.
Our  numerical results show that the proposed DAC algorithm is {\bf applicable} to the above LASSO model and it has a superior performance comparing to some popular decentralized methods, including NIDS and PG-EXTRA.
Here we remark that  the local objective functions
\begin{align*}
f_{i}(\vx)=\frac{1}{2} \bigg[\sum_{j\in B(i,m)} H(i,j)x(j)-b(i)\bigg]^2+\mu |x(i)|,\   i\in V,
\end{align*}
in the above LASSO model are not differentiable and hence
\Cref{ConvergenceTheorem} can not be applied to guarantee the exponential convergence
of the proposed DAC algorithm in the above LASSO model.

For both NIDS and PG-EXTRA, the settings of fusion centers are the same as DGD in the least squares problem, and the local  objective function at each fusion center is:
\begin{align*}
F_\lambda(\vx)= g_\lambda(\vx) + h_\lambda(\vx), \quad \lambda\in \Lambda,
\end{align*}
where $g_\lambda(\vx) = \frac{1}{2}\|\mH_{D_\lambda}\vx - \vb_{D_\lambda}\|$ and $h_\lambda(\vx) = \frac{\mu}{\# \Lambda}\|\vx\|_1$. The iteration scheme for NIDS is
\begin{align*}
\mX^{n+1} &= \Prox_{\alpha h} (\mZ^n)\\
\mZ^{n+1} &= \mZ^n - \mX^{n+1} + \Big(\mI - \frac{\beta}{2}(\mI - \mW) \Big) \Big(2\mX^{n+1} - \mX^n + \alpha (\nabla g(\mX^n) - \nabla g(\mX^{n+1})) \Big),
\end{align*}
and the iteration scheme for PG-EXTRA is
\begin{align*}
\mX^{n+1} &= \Prox_{\alpha h} (\mZ^n)\\
\mZ^{n+1} &= \mZ^n - \mX^{n+1} + \frac{\mI + \mW}{2} \big(2\mX^{n+1} - \mX^n\big) + \gamma (\nabla g(\mX^n) - \nabla g(\mX^{n+1})) ,
\end{align*}
where $\Prox_{\alpha h} (\mZ^n) = ((\Prox_{\alpha h_\lambda} (\vz^n_\lambda))^T)_{\lambda\in \Lambda}$, $\nabla g(\mX^n) = (\nabla^T g_\lambda(\mX^n))_{\lambda\in \Lambda}$
and $\mW$ is the same Metropolis mixing matrix as in \cref{eq:mixingmatrix}. The step sizes for NIDS are $\alpha=1.99(\max_{\lambda\in \Lambda} \|\mH_{D_\lambda}\|^2)^{-1}$ and $\beta=\frac{2}{\alpha}$, while the step size for PG-EXTRA is $\gamma=\alpha/2$.
Shown in  \Cref{fig:lasso} are the numerical results to solve the LASSO model
\eqref{lasso} via the DAC, NIDS and PG-EXTRA  algorithms. It is observed that our proposed DAC method converges faster than both PG-EXTRA and NIDS do, and
the  computational cost of DAC is also close to be linear with respect to the size of networks, which makes it more scalable than the other two do.

\begin{figure}[t] %[!h]
\begin{subfigure}[t]{0.48\textwidth}
\includegraphics[width = \textwidth]{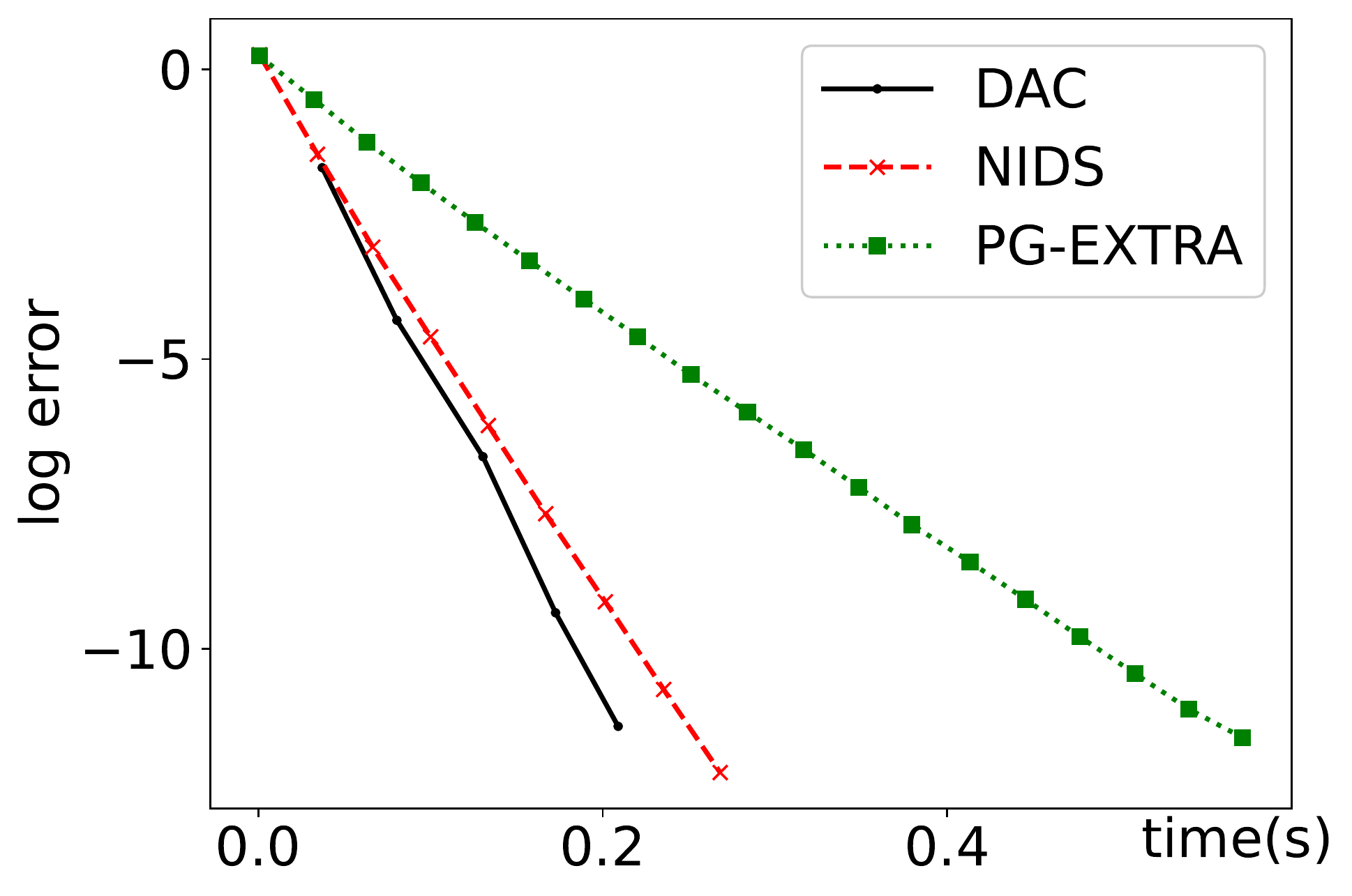}
\caption{$N=256$}
\end{subfigure}
\begin{subfigure}[t]{0.48\textwidth}
\includegraphics[width = \textwidth]{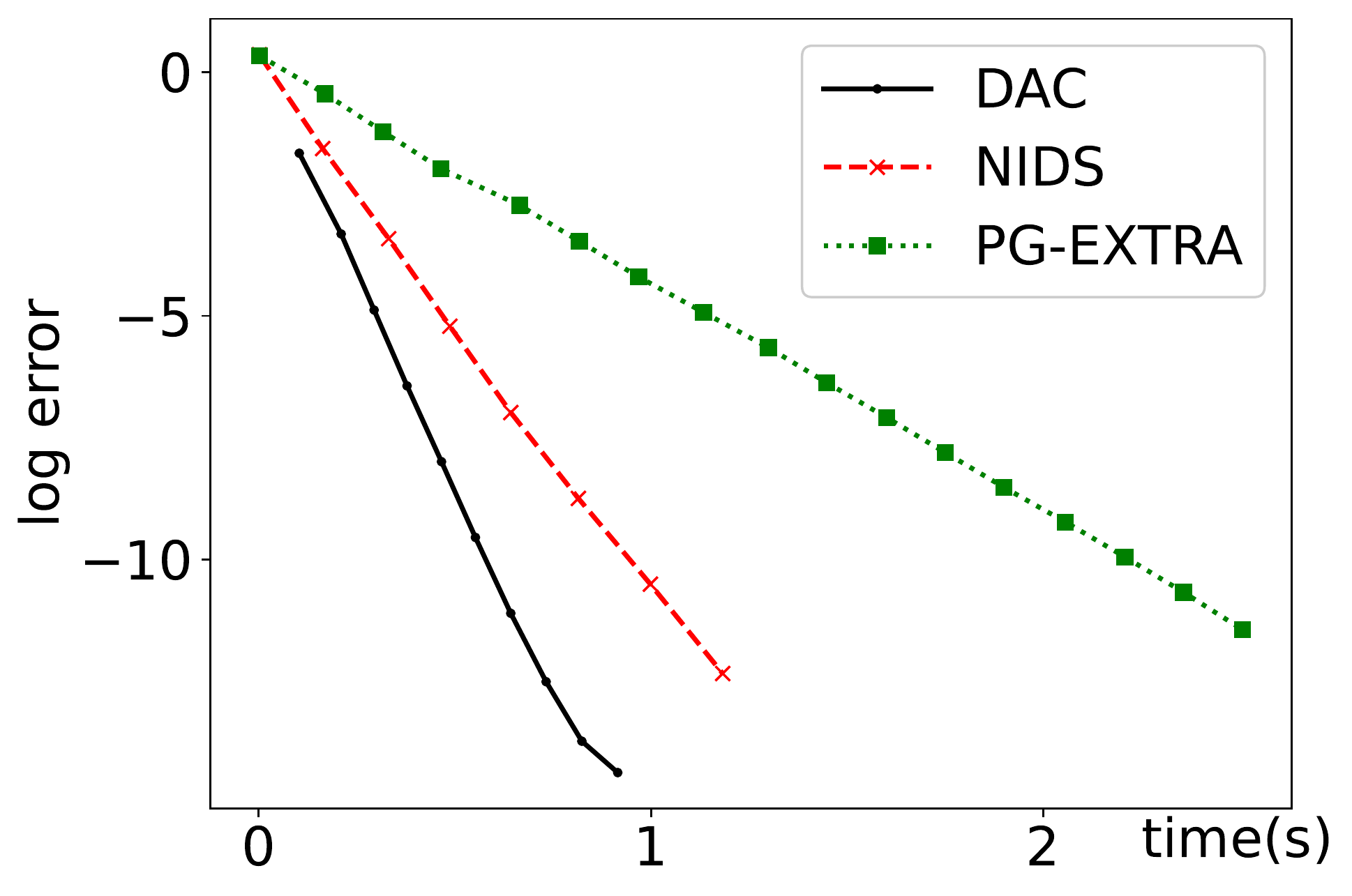}
\caption{$N=512$}
\end{subfigure}\\
\begin{subfigure}[t]{0.48\textwidth}
\includegraphics[width = \textwidth]{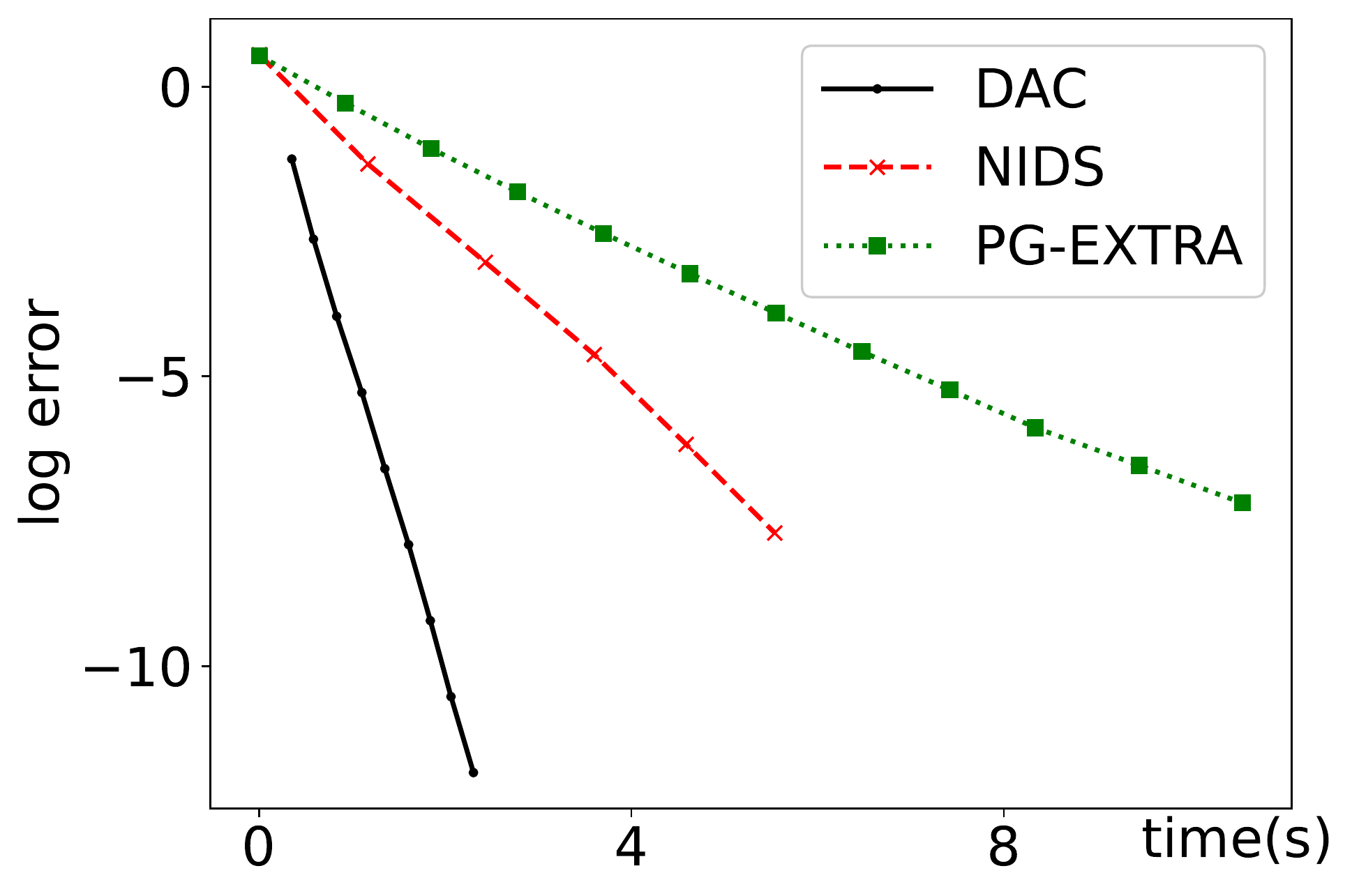}
\caption{$N=1024$}
\end{subfigure}
\begin{subfigure}[t]{0.48\textwidth}
\includegraphics[width = \textwidth]{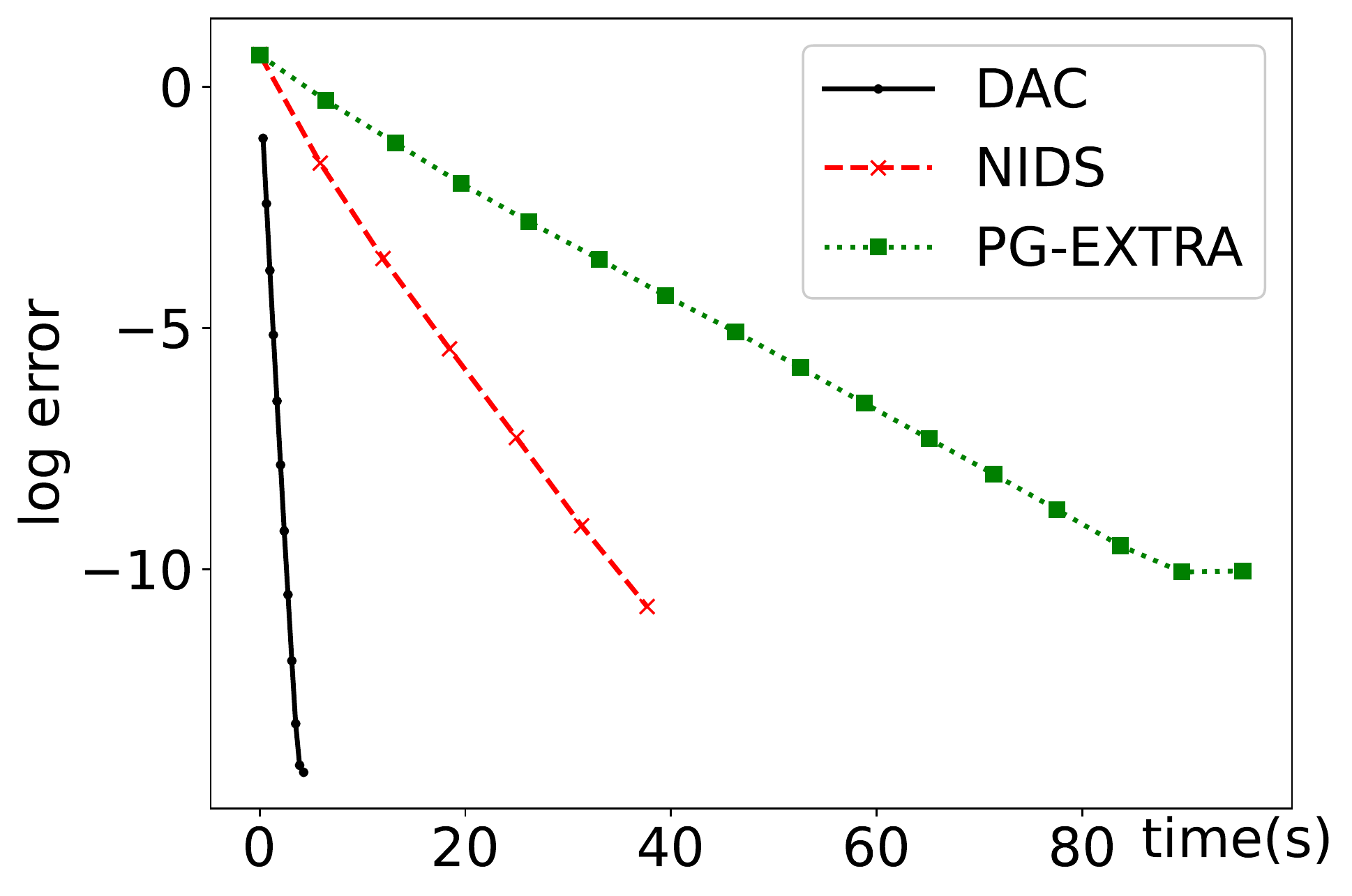}
\caption{$N=2048$}
\end{subfigure}
\caption{The logarithm error over running time (seconds) for the LASSO model \eqref{lasso}  on random geometric graphs of size $N$}
\label{fig:lasso}
\end{figure}

\section{Proofs}\label{proof.section}

In this section, we collect the proofs of  Theorems \ref{ConvergenceTheorem} and
\ref{inexactConvergenceTheorem}.

\subsection{Proof of Theorem \ref{ConvergenceTheorem}}\label{convergencetheorem.proofsection}
To prove Theorem \ref{ConvergenceTheorem}, we need two technical lemmas.
First we follow  the argument used in the proof of
Theorem IV.4 in \cite{Jiang19} to show that the inverse of
 a positive definite matrix  with limited geodesic-width has exponential off-diagonal decay.
The well-localization for the inverse of matrices
 is of great importance in
applied harmonic analysis, numerical analysis, distributed optimization and
many mathematical and engineering fields, see \cite{Grochenig2010a, Krishtal2011, Shin19, Shin2013, Fang2021}
for historical remarks and recent advances.

\begin{lemma}\label{Offdiagonaldecay.lemma} Let  $\cG=(V, E)$ be a connected undirected graph, and
matrices ${\mA}, {\bf B}$ on the graph $\cG$  satisfy
   \begin{equation}\label{Offdiagonaldecay.lemma.eq1}
\omega({\mA})\le \omega \ \ {\rm and}\  \
c{\bf I}\preceq{\mA}\preceq L {\mI}
\end{equation}
and
   \begin{equation*}
\omega({\mB})\le \omega \ \ {\rm and}\  \
\|{\mB}\|_{{\mathcal B}_2}\le M,
\end{equation*}
where $\omega$ is a positive integer and $c, L, M$ are positive constants with $0<c<L$.
Then  ${\mA}^{-1}=(G_1(i,j))_{i,j\in V}$ and ${\mA}^{-1} {\mB}=(G_2(i,j))_{i,j\in V}$ have exponential off-diagonal decay,
\begin{equation*}
|G_1(i,j)|\le  \frac{1}{c} \Big(1-\frac{c}{L}\Big)^{\rho(i,j)/\omega} \ \  {\rm and} \ \
|G_2(i,j)|\le \frac{LM}{c(L-c)} \Big(1-\frac{c}{L}\Big)^{\rho(i,j)/\omega},\ \ i, j\in V.
\end{equation*}
\end{lemma}

\begin{proof}
Set  ${\bf D}={\bf I}-L^{-1} {\bf A}$.
 Then it follows from \eqref{Offdiagonaldecay.lemma.eq1} that
\begin{equation}\label{Offdiagonaldecay.lemma.pfeq1}
\omega({\bf D})\le \omega \ \  {\rm and} \ \ \|{\bf D}\|_{\mathcal{B}_2}\le 1-c/L.
 \end{equation}
Hence
\begin{equation} \label{Offdiagonaldecay.lemma.pfeq2}
{\bf A}^{-1}= L^{-1} \sum_{n=0}^\infty {\bf D}^n,
\end{equation}
and
\begin{equation} \label{Offdiagonaldecay.lemma.pfeq3}
\omega ({\bf D}^n)\le n \omega \ \ {\rm and}\ \  \omega ({\bf D}^n{\bf B})\le (n+1) \omega , \ n\ge 0.\end{equation}
Take $i,j\in V$, let $n_0(i,j)$ be the smallest nonnegative integer  such that $\rho(i,j)/\omega \le n_0(i,j)$, and write  ${\bf D}^n=(D_n(i,j))_{i,j\in V}, n\ge 0$.
Then we obtain  from \eqref{NormInequality},
\eqref{Offdiagonaldecay.lemma.pfeq1}, \eqref{Offdiagonaldecay.lemma.pfeq2} and \eqref{Offdiagonaldecay.lemma.pfeq3} that
\begin{eqnarray*}
|G_1(i,j)|&\hskip-0.08in = & \hskip-0.08in L^{-1}\Big|\sum_{n=0}^{\infty} D_n(i,j)\Big|= L^{-1}\Big|\sum_{n=n_0(i,j)}^{\infty} D_n(i,j)\Big| \\
&\hskip-0.08in\le &\hskip-0.08in L^{-1} \sum_{n=n_0(i,j)}^{\infty} \|{\bf D}^n\|_{\infty} \le L^{-1}\sum_{n=n_0(i,j)}^{\infty} \|{\bf D}^n\|_{\mathcal{B}_2}\le L^{-1}\sum_{n=n_0(i,j)}^{\infty} \|{\bf D}\|^n_{\mathcal{B}_2}
\\
& \hskip-0.08in \le  & \hskip-0.08in L^{-1} \sum_{n=n_0(i,j)}^{\infty} \Big(1-\frac{c}{L}\Big)^n
=\frac{ 1}{c}\Big(1-\frac{c}{L}\Big)^{n_0(i,j)}
\le \frac{1}{c} \Big(1-\frac{c}{L}\Big)^{\rho(i,j)/\omega}
\end{eqnarray*}
and similarly
\begin{eqnarray*}
|G_2(i,j)|&\hskip-0.08in \le & \hskip-0.08in
  L^{-1}\sum_{n=n_0(i,j)-1}^{\infty} \|{\bf D}^n{\bf B}\|_{\mathcal{B}_2}
\le \frac{LM}{c(L-c)} \Big(1-\frac{c}{L}\Big)^{\rho(i,j)/\omega}, \ \ i,j\in V.
\end{eqnarray*}
This completes the proof.
\end{proof}

To prove the Theorem \ref{ConvergenceTheorem}, we also need the following lemma about the summation of an  exponential decay sequence.
\begin{lemma}
 Let $\cG=(V, E)$ be as in Theorem \ref{ConvergenceTheorem}.
Then for all  $R\ge 1$ and $\beta>0$,
\begin{equation}  \label{exponentialdecay.lemma.eq1}
\sum_{j\in V \ {\rm with}\ \rho(i,j)\ge R} e^{-\beta  \rho(i,j)}\le
 {D_1(\cG) d(\cG)! } \beta^{-d(\cG)}   (R+1)^{d({\mathcal G})}e^{-\beta(R-1)}, \ \ i\in V.
\end{equation}
\end{lemma}

\begin{proof}
We follow the arguments in \cite{Cheng17, Jiang19, Shin19}.  Take $i\in V$.
For any $0<\sigma<1$, we have
\begin{eqnarray*}
\sum_{j\in V \ {\rm with}\ \rho(i,j)\ge R} e^{-\beta \rho(i,j)}& \hskip-0.08in \le & \hskip-0.08in  \sum_{k\in {\mathbb Z}\ {\rm with}\  k\sigma\ge R} e^{-(k-1)\beta \sigma}\mu_{\cG}\big(B(i, k\sigma)\backslash B(i, (k-1)\sigma)\big)\nonumber\\
& \hskip-0.08in \le & \hskip-0.08in  \sum_{k\in {\mathbb Z}\ {\rm with}\  k\sigma\ge R} (e^{-(k-1)\beta \sigma}- e^{-k\beta \sigma})\mu_{\cG}\big(B(i, k\sigma)\big)\nonumber\\
& \hskip-0.08in \le & \hskip-0.08in \beta \sigma  D_1(\cG) \sum_{k\in {\mathbb Z}\ {\rm with}\  k\sigma\ge R}  e^{-(k-1)\beta\sigma} (k\sigma+1)^{d(\cG)}.
\end{eqnarray*}
Taking limit $\sigma\to 0$ in the above estimate yields
\begin{eqnarray*}
\sum_{j\in V \ {\rm with}\ \rho(i,j)\ge R} e^{-\beta\rho(i,j)}
& \hskip-0.08in \le &\hskip-0.08in  \beta  D_1(\cG) \int_R^\infty (t+1)^{d(\cG)} \exp (-\beta t ) dt\nonumber\\
& \hskip-0.08in \le  &\hskip-0.08in \beta D_1(\cG)  (R+1)^{d({\mathcal R})} e^{-\beta(R-1)} \times
\int_1^\infty  s^{d(\cG)} e^{-\beta s} ds\nonumber\\
& \hskip-0.08in \le  &\hskip-0.08in   {D_1(\cG) d(\cG)! } {\beta^{-d(\cG)}  } (R+1)^{d({\mathcal G})} e^{-\beta (R-1)},
\end{eqnarray*}
where we substitute $t$ by $s=t-R+1$ in the second inequality and apply the inequality $t+1=s+R\le s (R+1)$.
This proves \eqref{exponentialdecay.lemma.eq1}.
\end{proof}

\begin{proof} [Proof of Theorem \ref{ConvergenceTheorem}]
Let $\vw_\lambda^n, \lambda\in \Lambda$, be as in \eqref{DACalgorithm.defa}.
By \eqref{GlobalHessianBounds.eq1}  and
 \eqref{ConvergenceTheorem.pfeq0}, we have for $\lambda\in \Lambda$
 \begin{eqnarray}\label{ConvergenceTheorem.pfeq2}
& \hskip-0.08in &  \hskip-0.08in \chi_{D_{\lambda, R}} \nabla F(\mI_{D_{\lambda,R}}\hat \vx+\mI_{V\backslash D_{\lambda,R}} \vx^{n})\nonumber\\
& \hskip-0.08in = &  \hskip-0.08in \chi_{D_{\lambda, R}}\big( \nabla F(\mI_{D_{\lambda,R}}\hat \vx+\mI_{V\backslash D_{\lambda,R}} \vx^{n})-
\nabla F(\chi^*_{D_{\lambda,R}} \vw_{\lambda}^n+\mI_{V\backslash D_{\lambda,R}} \vx^{n})\big)\nonumber\\
\quad & \hskip-0.08in = & \hskip-0.08in
\chi_{D_{\lambda, R}} {\bf \Phi}\big(\mI_{D_{\lambda,R}}\hat \vx+\mI_{V\backslash D_{\lambda,R}} \vx^{n},\chi^*_{D_{\lambda,R}} \vw_{\lambda}^n+\mI_{V\backslash D_{\lambda,R}} \vx^{n}\big)
  \mI_{D_{\lambda,R}} ( \hat\vx-\chi^*_{D_{\lambda,R}}\vw_{\lambda}^n).
\end{eqnarray}
For $\lambda\in \Lambda$, define
\begin{equation*} \label{ConvergenceTheorem.pfeq3}
{\bf \Phi}_{\lambda, R,  n}={\bf I}_{D_{\lambda, R}} {\bf \Phi}\big(\mI_{D_{\lambda,R}}\hat \vx+\mI_{V\backslash D_{\lambda,R}} \vx^{n}, \chi^*_{D_{\lambda,R}} \vw_{\lambda}^n+\mI_{V\backslash D_{\lambda,R}} \vx^{n}\big)
 {\bf I}_{D_{\lambda,R}}+ \frac{L+c}{2} {\bf I}_{V\backslash D_{\lambda, R}}.
\end{equation*}
Then it  follows from
\eqref{GlobalHessianBounds.eq2} and \eqref{ConvergenceTheorem.pfeq2} that
\begin{equation} \label{ConvergenceTheorem.pfeq4-}
\omega({\bf \Phi}_{\lambda, R, n})\le 2m, \ \    c {\bf I}\le {\bf \Phi}_{\lambda, R, n}\le L
 {\bf I},
\end{equation}
and
\begin{equation} \label{ConvergenceTheorem.pfeq4}
  \mI_{D_{\lambda,R}}(\chi^*_{D_{\lambda,R}} \vw_{\lambda}^n- \hat\vx)= -  \mI_{D_{\lambda,R}} ({\bf \Phi}_{\lambda, R,  n})^{-1} \mI_{D_{\lambda, R}} \nabla F(\mI_{D_{\lambda,R}}\hat \vx+\mI_{V\backslash D_{\lambda,R}} \vx^{n}), \ \ \lambda\in \Lambda.
 \end{equation}

By \eqref{GlobalHessianBounds.eq2} and \eqref{GlobalHessianBounds.eq1}, we conclude that  the objective function $F$ is strictly convex, and hence the solution $\hat \vx={\rm arg}\min F(\vx)$ satisfies
  \begin{equation}\label{ConvergenceTheorem.pfeq5}
 \nabla F({\hat \vx})={\bf 0}.
\end{equation}
Following the argument used in \eqref{ConvergenceTheorem.pfeq2}
and applying \eqref{GlobalHessianBounds.eq1}  and \eqref{ConvergenceTheorem.pfeq5}, we obtain
 \begin{eqnarray}\label{ConvergenceTheorem.pfeq6}
\chi_{D_{\lambda, R}} \nabla F(\mI_{D_{\lambda,R}}\hat \vx+\mI_{V\backslash D_{\lambda,R}} \vx^{n})
=\chi_{D_{\lambda, R}}
{\bf \Phi}(\mI_{D_{\lambda,R}}\hat \vx+\mI_{V\backslash D_{\lambda,R}} \vx^{n},
\hat \vx)\mI_{V\backslash D_{q,R}} (\vx^n-\hat\vx).
\end{eqnarray}

 Combining \eqref{DACalgorithm.defb}, \eqref{ConvergenceTheorem.pfeq4} and \eqref{ConvergenceTheorem.pfeq6} yields
\begin{equation} \label{ConvergenceTheorem.pfeq7}
\vx^{n+1}-\hat\vx  = \sum_{\lambda\in \Lambda}\mI_{D_{\lambda}}(\chi^*_{D_{\lambda,R}}\vw_{\lambda}^n-\hat\vx)
=: {\bf H}_n  (\vx^n-\hat\vx),
\end{equation}
where
\begin{equation}\label{ConvergenceTheorem.pfeq8}
{\bf H}_n= -\sum_{\lambda\in \Lambda}  {\bf I}_{D_{\lambda}}
 ({\bf \Phi}_{\lambda, R,  n})^{-1} {\bf I}_{D_{\lambda, R}}
{\bf \Phi}(\mI_{D_{\lambda,R}}\hat \vx+\mI_{V\backslash D_{\lambda,R}} \vx^{n},
\hat \vx){\bf I}_{V\backslash D_{\lambda,R}},\  n\ge 0.\end{equation}

\smallskip
By \eqref{GlobalHessianBounds.eq2}  and \eqref{NormInequality},
we have
\begin{equation} \label{ConvergenceTheorem.pfeq8}
\omega\big({\bf I}_{D_{\lambda, R}} {\bf \Phi}(\mI_{D_{\lambda,R}}\hat \vx+\mI_{V\backslash D_{\lambda,R}} \vx^{n},
\hat \vx)\big)\le 2m \ \ {\rm and} \ \ \big\|{\bf I}_{D_{\lambda, R}} {\bf \Phi}(\mI_{D_{\lambda,R}}\hat \vx+\mI_{V\backslash D_{q,R}} \vx^{n},
\hat \vx)\big\|_{{\mathcal B}_2} \le L.
\end{equation}
Write ${\bf H}_n=(H_n(i,j))_{i,j\in V}$. Then for any $i,j\in V$  we get from \eqref{governingvertices.def}, \eqref{extendedneighbor.def}, \eqref{ConvergenceTheorem.pfeq4-}, \eqref{ConvergenceTheorem.pfeq8} and  Lemma \ref{Offdiagonaldecay.lemma}  that
\begin{eqnarray} \label{ConvergenceTheorem.pfeq9}
|H_n(i,j)| & \hskip-0.08in \le &  \hskip-0.08in \frac{L^2}{c(L-c)} \Big(1-\frac{c}{L}\Big)^{\rho(i,j)/(2m)} \times \sum_{\lambda\in \Lambda}  \mathbbm{1}_{D_\lambda}(i) \mathbbm{1}_{V\backslash D_{\lambda, R}}(j)\nonumber\\
&  \hskip-0.08in \le &  \hskip-0.08in \left\{\begin{array}{ll} \frac{L^2}{c(L-c)} \Big(1-\frac{c}{L}\Big)^{\rho(i,j)/(2m)}
& {\rm if}\ \rho(i,j)>R\\
0 & {\rm if}\ \rho(i,j)\le R,
\end{array}
\right.
 \end{eqnarray}
 where $\mathbbm{1}_E$ is the indicator function on a set $E$.
Let $\delta_R$ be as in \eqref{deltaR.def}.  Combining  \eqref{exponentialdecay.lemma.eq1} and \eqref{ConvergenceTheorem.pfeq9} yields
\begin{equation} \label{ConvergenceTheorem.pfeq17}
\|{\bf H}_n\|_{\mathcal S} \le
\sup_{i\in V} \sum_{\rho(i,j)>R}  \frac{L^2}{c(L-c)} \Big(1-\frac{c}{L}\Big)^{\rho(i,j)/(2m)}\le \delta_R.
 \end{equation}

By \eqref{NormInequality}, \eqref{ConvergenceTheorem.pfeq7} and
\eqref{ConvergenceTheorem.pfeq17}, we obtain
\begin{equation}\label{ConvergenceTheorem.pfeq18}
\|\vx^{n+1}-\hat\vx \|_p\le \|{\bf H}_n \|_{\mathcal S} \|\vx^{n+1}-\hat\vx \|_p\le \delta_R \|\vx^{n}-\hat\vx \|_p, \ n\ge 0.
\end{equation}
Applying \eqref{ConvergenceTheorem.pfeq18} iteratively proves \eqref{ConvergenceTheorem.eq4} and hence completes the proof.
\end{proof}

\subsection{Proof of Theorem \ref{inexactConvergenceTheorem}}
\label{inexactConvergenceTheorem.proofsection}

Let $\tilde \vw_\lambda^n, \lambda\in \Lambda$, be as in \eqref{inexactDACalgorithm.defa}.
By \eqref{GlobalHessianBounds.eq1}  and
\eqref{ConvergenceTheorem.pfeq5}, we have
 \begin{eqnarray}\label{inexactConvergenceTheorem.pfeq1}
& \hskip-0.08in &  \hskip-0.08in \chi_{D_{\lambda, R}} \nabla F(\mI_{D_{\lambda,R}}\hat \vx+\mI_{V\backslash D_{\lambda,R}} \tilde \vx^{n})\nonumber\\
& \hskip-0.08in = &  \hskip-0.08in \chi_{D_{\lambda, R}}\big( \nabla F(\mI_{D_{\lambda,R}}\hat \vx+\mI_{V\backslash D_{\lambda,R}} \tilde \vx^{n})-
\nabla F(\chi^*_{D_{\lambda,R}} \tilde \vw_{\lambda}^n+\mI_{V\backslash D_{\lambda,R}} \tilde \vx^{n})\big)\nonumber\\
& & +\chi_{D_{\lambda, R}} \nabla F(\chi^*_{D_{\lambda,R}} \tilde \vw_{\lambda}^n+\mI_{V\backslash D_{\lambda,R}}\tilde \vx^{n})\nonumber\\
\quad & \hskip-0.08in = & \hskip-0.08in
\chi_{D_{\lambda, R}} {\bf \Phi}\big(\mI_{D_{\lambda,R}}\hat \vx+\mI_{V\backslash D_{\lambda,R}} \tilde \vx^{n}, \chi^*_{D_{\lambda,R}} \tilde \vw_{\lambda}^n+\mI_{V\backslash D_{\lambda,R}} \tilde \vx^{n}\big)
 \mI_{D_{\lambda,R}} (\hat\vx- \chi^*_{D_{\lambda,R}} \tilde \vw_{\lambda}^n)\nonumber\\
& & +\chi_{D_{\lambda, R}} \nabla F(\chi^*_{D_{\lambda,R}} \tilde \vw_{\lambda}^n+\mI_{V\backslash D_{\lambda,R}}\tilde \vx^{n})
\end{eqnarray}
and
 \begin{equation}\label{inexactConvergenceTheorem.pfeq2}
 \chi_{D_{\lambda, R}} \nabla F(\mI_{D_{\lambda,R}}\hat \vx+\mI_{V\backslash D_{\lambda,R}} \tilde \vx^{n})
 =
\chi_{D_{\lambda, R}}
{\bf \Phi}(\mI_{D_{\lambda,R}}\hat \vx+\mI_{V\backslash D_{\lambda,R}} \tilde \vx^{n},
\hat \vx)\mI_{V\backslash D_{\lambda,R}} (\tilde \vx^n-\hat\vx),  \ \lambda\in \Lambda.
\end{equation}
Define
\begin{equation*}
\tilde {\bf \Phi}_{\lambda, R,  n}={\bf I}_{D_{\lambda, R}} {\bf \Phi}\big(\mI_{D_{\lambda,R}}\hat \vx+\mI_{V\backslash D_{\lambda,R}}\tilde \vx^{n}, \chi^*_{D_{\lambda,R}} \tilde\vw_{\lambda}^n+\mI_{V\backslash D_{\lambda,R}} \tilde \vx^{n}\big)
 {\bf I}_{D_{\lambda,R}}+ \frac{L+c}{2} {\bf I}_{V\backslash D_{\lambda, R}}, \ \ \lambda\in \Lambda,
\end{equation*}
and
\begin{equation*}
\tilde {\bf H}_n= -\sum_{\lambda\in \Lambda}  {\bf I}_{D_{\lambda}}
 (\tilde {\bf \Phi}_{\lambda, R,  n})^{-1} {\bf I}_{D_{\lambda, R}}
{\bf \Phi}(\mI_{D_{\lambda,R}}\hat \vx+\mI_{V\backslash D_{\lambda,R}} \tilde \vx^{n},
\hat \vx){\bf I}_{V\backslash D_{\lambda,R}},\  n\ge 0.
\end{equation*}
By \eqref{inexactDACalgorithm.defb}, \eqref{inexactConvergenceTheorem.pfeq1}
and \eqref{inexactConvergenceTheorem.pfeq2}, we obtain
\begin{equation} \label{inexactConvergenceTheorem.pfeq5}
\tilde \vx^{n+1}-\hat\vx
=\tilde {\bf H}_n  (\tilde\vx^n-\hat\vx)+\sum_{\lambda\in \Lambda} \mI_{D_\lambda} \big (\tilde {\bf \Phi}_{\lambda, R,  n}\big)^{-1}
\mI_{D_{\lambda, R}} \nabla F(\chi^*_{D_{\lambda,R}} \tilde \vw_{\lambda}^n+\mI_{V\backslash D_{\lambda,R}}\tilde \vx^{n})
\end{equation}
Following the argument used in the proof of Theorem \ref{ConvergenceTheorem}, we can show that
\begin{equation} \label{inexactConvergenceTheorem.pfeq6}
\omega(\tilde {\bf \Phi}_{\lambda, R, n})\le 2m \ \ {\rm and}\  \   c {\bf I}\le \tilde {\bf \Phi}_{\lambda, R, n}\le L{\bf I}, \ \lambda\in \Lambda,
\end{equation}
and
\begin{equation} \label{inexactConvergenceTheorem.pfeq7}
\|{\bf H}_n\|_{\mathcal S}\le \delta_R.
\end{equation}
Write  $ \big (\tilde {\bf \Phi}_{\lambda, R,  n}\big)^{-1}=(G_1(i,j))_{i,j\in V}$. Then  from
 \eqref{inexactConvergenceTheorem.pfeq6} and Lemma \ref{Offdiagonaldecay.lemma} we obtain
  \begin{equation*}
  |G_1(i,j)|\le \frac{1}{c} \Big(1-\frac{c}{L}\Big)^{\rho(i,j)/(2m)},\  i, j\in V.
  \end{equation*}
  This together with \eqref{governingvertices.def}, \eqref{extendedneighbor.def} and \eqref{inexactDACalgorithm.defa} implies that
\begin{eqnarray}  \label{inexactConvergenceTheorem.pfeq8}
& \hskip-0.08in  & \hskip-0.08in \Big\|  \sum_{\lambda\in \Lambda} \mI_{D_\lambda} \big (\tilde {\bf \Phi}_{\lambda, R,  n}\big)^{-1}
\mI_{D_{\lambda, R}} \nabla F(\chi^*_{D_{\lambda,R}} \tilde \vw_{\lambda}^n+\mI_{V\backslash D_{\lambda,R}}\tilde \vx^{n})\Big\|_\infty\nonumber\\
& \hskip-0.08in \le  & \hskip-0.08in \frac{\epsilon_n}{c}
\sup_{i\in V} \sum_{\lambda\in \Lambda}\sum_{j\in V} \mathbbm{1}_{D_\lambda}(i)
 \Big(1-\frac{c}{L}\Big)^{\rho(i,j)/(2m)} \mathbbm{1}_{D_{\lambda, R}}(j)\le  \frac{L-c}{L^2} \delta_R \epsilon_n, \ n\ge 0.
\end{eqnarray}

Combining \eqref{NormInequality}, \eqref{inexactConvergenceTheorem.pfeq5}, \eqref{inexactConvergenceTheorem.pfeq7}and \eqref{inexactConvergenceTheorem.pfeq8}, we get
\begin{equation}   \label{inexactConvergenceTheorem.pfeq9}
\|\tilde \vx^{n+1}-\hat x\|_\infty\le \delta_R \|\tilde \vx^n-\hat \vx\|_\infty+ \frac{L-c}{L^2} \delta_R \epsilon_n,\  n\ge 0.
\end{equation}
Applying \eqref{inexactConvergenceTheorem.pfeq9} repeatedly proved the desired conclusion \eqref{inexactConvergenceTheorem.eq1}.

\bibliographystyle{siam}

\bibliography{Distributed_Load_Optimization.bib}

\end{document}